\documentclass[12pt,leqno,twoside]{amsart}

\usepackage[latin1]{inputenc}
\usepackage[T1]{fontenc}
\usepackage[colorlinks=true, pdfstartview=FitV, linkcolor=blue, citecolor=blue, urlcolor=blue]{hyperref}
\usepackage{amstext,amsmath,amscd, bezier,indentfirst,amsthm,amsgen,enumerate, geometry}
\usepackage[all,knot,arc,import,poly]{xy}
\usepackage{amsfonts,color, soul}  
\usepackage{amssymb}
\usepackage{latexsym}
\usepackage{epsfig}
\usepackage{graphicx}
\usepackage{srcltx}
\usepackage{enumitem}

\usepackage{tikz-cd}
\usepackage[all]{xy}

\usepackage{tikz} 
\usetikzlibrary{through} 
\usetikzlibrary{patterns} 
\usetikzlibrary{intersections} 
\usetikzlibrary{matrix} 
\usetikzlibrary{cd} 

\usepackage[hyperpageref]{backref}

\renewcommand*{\backrefalt}[4]{
	\ifcase #1 %
	No citation.%
	\or
	Citation on page #2.%
	\else
	Citation on pages #2.%
	\fi}

\newtheorem{theorem}{Theorem}[section]
\newtheorem{corollary}[theorem]{Corollary}
\newtheorem{lemma}[theorem]{Lemma}
\newtheorem{proposition}[theorem]{Proposition}
\theoremstyle{definition}
\newtheorem{definition}[theorem]{Definition}
\theoremstyle{remark}
\newtheorem{remark}[theorem]{\sc Remark}
\newtheorem{example}[theorem]{\sc Example}
\newtheorem{question}[theorem]{\sc Question}

\newcommand{\Sing}{{\rm{Sing\hspace{2pt}}}}

\newcommand{\rank}{{\rm{rank\hspace{2pt}}}}

\newcommand{\Disc}{{\rm{Disc\hspace{2pt}}}}

\renewcommand{\d}{{\rm{d}}}

\renewcommand{\aa}{{\rm{a}}}

\newcommand{\e}{\varepsilon}
\newcommand{\dd}{{\rm{d}}}
\newcommand{\m}{\setminus}

\newcommand{\cT}{{\mathcal T}}

\newcommand{\bR}{{\mathbb R}}
\newcommand{\bC}{{\mathbb C}}

\newcommand{\bW}{{\mathbb W}}

\newcommand{\R}{\mathbb{R}}
\newcommand{\C}{\mathbb{C}}

\def\dsum{\displaystyle \sum}

\def\dlim{\displaystyle \lim}

\begin{document}
	
	\title{Some remarks about $ \rho $-regularity for real analytic maps}

	\author{Maico Ribeiro, Ivan Santamaria and  Thiago da Silva}
	
	\address{UFES,Universidade Federal do Espírito Santo,  Av. Fernando Ferrari, 514 -CEP 29.075-910 Vitória, Espirito Santo,  Brazil}
	\email{maico.ribeiro@ufes.br}
	
	\address{UFES,Universidade Federal do Espírito Santo,  Av. Fernando Ferrari, 514 -CEP 29.075-910 Vitória, Espirito Santo,  Brazil}
	\email{thiago.silva@ufes.br}
	
	\address{UFES,Universidade Federal do Espírito Santo,  Av. Fernando Ferrari, 514 -CEP 29.075-910 Vitória, Espirito Santo,  Brazil}
	\email{ivansantamaria@ufes.br}

	\begin{abstract}
		In this paper,  we discuss the concept of $\rho$-regularity of analytic map germs and its close relationship with the existence of locally trivial smooth fibrations, known as the Milnor fibrations. The presence of a Thom regular stratification or the Milnor condition (b) at the origin, indicates the transversality of the fibers of the map $G$ with respect to the levels of a function $\rho$, which guarantees $\rho$-regularity. Consequently, both conditions are crucial for the presence of \textit{open book} structures and the Milnor fibrations. The work aims to provide a comprehensive overview of the main results concerning the existence of Thom regular stratifications and the Milnor condition (b) for germs of analytic maps. It presents strategies and criteria to identify and ensure these regularity conditions and discusses situations where they may not be satisfied. The goal is to understand the presence and limitations of these conditions in various contexts.
	\end{abstract}
	
	\maketitle

	\section{Introduction}

	The transversality of the fibers of a map germ $G:(\mathbb{K}^{m},0) \rightarrow (\mathbb{K}^{p}, 0)$, $m\ge p >1$, $\mathbb{K} = \bR$ or $\bC$, at the levels of a function $\rho$, which defines the origin, is called \textit{$\rho$-regularity} of the map. The $\rho$-regularity is a condition for the existence of locally trivial smooth fibration and has been used in the local (stratified) sense by Thom, Milnor, Mather, Looijenga, Bekka, e.g. \cite{Be,Lo,Mi,Th1,Th2}, and more recently in \cite{dST1,ACT}, and \cite{CSS, CSS2}, and also employed in the global case in the references \cite{ACT-inf,DRT,NZ,Ti2,Ti3}. 
	
	\vspace{0.3cm}
	
	The close relationship between $\rho$-regularity and the existence of fibered structures defined by a map germ $G$ and consequently, with the study of topological properties of maps and singularities of associated spaces, has led to the need to create strategies to find natural and convenient tools to ensure such condition to produce interesting classes of maps $G$ with  Milnor fibration structures.
	
	\vspace{0.3cm}
	
	The search for such tools has resulted in a series of works published since the 1960s, for example, \cite{HL,HMS,Mi,Sa1,Sa2,Le} and \cite{ACT-inf,dST1,ACT,CSS,CSS2,DRT,Ma,Mi,NZ,Th1,Th2,Ti2,Ti3} and more recently, \cite{ART,ADRS,CT23,JT1,JT2,PT,Ri}.
	
	\vspace{0.3cm}
	
	In \cite{HL} and \cite{Le}, Hamm and Lê proved that the existence of a \textit{Thom regular} stratification of some holomorphic functions $f: (\mathbb{C}^n, 0) \to (\mathbb{C}, 0)$, with $n > 1$, ensures the existence of the Milnor tube fibration. For the case of germs of holomorphic or real analytic maps   $G: (\mathbb{K}^{m}, 0) \rightarrow (\mathbb{K}^{p}, 0)$,  recent results have shown that the existence of the Milnor fibration in the tube can also be guaranteed by the existence of a Thom regular stratification or by the absence of points from a set called the \textit{Milnor set} of the map $G$, in a conical neighborhood of the manifold $V_G:=G^{-1}(0)$. The latter condition is known as the \textit{Milnor condition (b)} at the origin.
	
	\vspace{0.3cm}

	Such results have made clear the close relationship between the existence of a Thom regular stratification or the Milnor condition (b) and the presence of fibered structures. In reality, the connection between the existence of these regularity conditions and the presence of a fibration structure associated with a germ of an analytic map is established through $\rho$-regularity. Both the existence of a Thom stratification and the Milnor condition (b) imply $\rho$-regularity.
	
	\vspace{0.3cm}
	
	The existence of a Thom stratification is a crucial indication of the transversality of the fibers of the map $G$ to the levels of the function $\rho$ that defines the origin. This is essential to guarantee the presence of "open book" structures locally, and in turn, the Milnor fibration on the tube and on the sphere.
	
	\vspace{0.3cm}
	
	On the other hand, the Milnor condition (b) also leads to $\rho$-regularity, providing an additional guarantee of the existence of a conical neighborhood of $V_G$ where there are no points from the Milnor set of $G$. This ensures the regularity of the map's fibers with respect to the levels of $\rho$.

	\vspace{0.3cm}

	This work aims to provide a detailed survey of the main and most recent results concerning the existence of Thom regular stratifications and the Milnor condition (b) at the origin for germs of analytic maps. We will address some new general conditions under which $G: (\mathbb{R}^m, 0) \to (\mathbb{R}^p, 0)$, with $m \geq p > 1$ and $\textrm{dim}\, \Disc G \geq 0$, is Thom regular or satisfies the Milnor condition (b) at the origin.
	
	\vspace{0.3cm}
	
	Throughout the work, we will present strategies and criteria that allow us to identify and guarantee the existence of Thom regular stratifications, as well as the fulfillment of the Milnor condition (b) and the relationship between these two regularity conditions. Additionally, we will also address situations where such conditions are not satisfied, which is equally relevant to understand the limitations of these structures in certain contexts.
	
	\vspace{0.3cm}

	The paper is organized as follows. In Section 2, we introduce the notion of $\rho$-regularity and demonstrate how this condition is directly related to the existence of fibration structures. In Section 3, we present the concept of Thom regularity for analytic map germs and show that this condition implies $\rho$-regularity. In Section 4, we provide several criteria to ensure Thom regularity for analytic map germs and include a series of examples to illustrate the presented results. In Section 5, we introduce the Thom Irregular set and present some novel criteria to guarantee when Thom regularity is not satisfied. We conclude the work by discussing the Milnor condition (b). We show that this condition also implies $\rho$-regularity. Furthermore, we present some recent criteria to ensure the Milnor condition (b) for germ compositions and finally, we compare the Milnor condition (b) with Thom regularity.

	
	\section{The $ \rho $-regularity and the existence of the Milnor tube fibrations}

	In this section, we present the concept of the Milnor tube fibration for a germ of a real analytic map and its relationship with the $ \rho $-regularity. The investigation of these fibrations has been a subject of extensive research by mathematicians since the 1980s. For more comprehensive details, you can refer to the works of \cite{ADER,Ma,Mi,Ri,Le}, and \cite{ART} for additional generalities.
	
	\vspace{0.2cm}
	
	Let $G:(\bR^m,0) \to (\bR^p,0),$ $m\geq p\geq 2,$ be a non-constant real analytic map germ and let $G: U \to \bR^{m},$ $G(0)=0$ be a representative of the germ, where $U \subseteq \mathbb{R}^{m}$ is an open set and $G(x)=(G_1(x),G_2(x), \ldots, G_p(x)).$ From now on we will denote a germ and its representative by the same notation.
	
	\vspace{0.3cm}
	
	Denote the zero locus $ G^{-1}(0)$ of $G$ by $ V_G$ and the singular set by $\Sing G:=\{x\in U: \rank(d G(x))~\text{is not maximal}\}$, where $\d G(x)$ denotes the Jacobian matrix of $G$ at $x$. The discriminant set of $G$ is given by $\Disc(G):=G(\Sing(G)).$ We say that $G$ has an isolated critical value at origin if $\Disc(G)=\{0\}$, that is, $ \Sing G \subseteq V_G $. If $ \Sing G \not\subseteq V_G $, then $ \Disc G $ is positive dimensional.
	
	\vspace{0.3cm}
	
	In this paper, all analytic map germs will be considered non-constant. 
	Moreover, as already mentioned above, without loss of generality, all sets and maps will be considered as a germ at the origin, once they are well defined. 
	
	\vspace{0.3cm}
	
	\begin{definition} We say that an analytic map germ $G:(\bR^m,0) \to (\bR^p,0)$ admits a Milnor's tube fibration if there exist $\epsilon_{0}>0$ and $\eta>0$ such that for all $0< \epsilon \leq \epsilon_{0}$, in the closed ball $\bar{B}^M_{\epsilon}$ the restriction map 
		\begin{equation}\label{tmil}
			G_{|}:\bar{B}^m_{\epsilon}\cap G^{-1}(\bar{B}_{\eta}^{p}-\{0\})\to \bar{B}_{\eta}^{p}-\{0\}.
		\end{equation}
		
		\noindent is a locally trivial smooth fibration for all $0<\eta \ll \epsilon.$ Moreover, the diffeomorphism type does not depend on the choice of $\epsilon$ and $\eta$.
		
	\end{definition}
	
	
	
	To state the existence of the tube fibration we need to introduce some standard notation and definitions. Let $U \subset \mathbb{R}^m$ be an open set, $0\in U$, and let $\rho: U \to \mathbb{R}_{\ge 0}$ be a smooth nonnegative and proper function that is zero only at the origin.

	\vspace{0.3cm}

	

	\vspace{0.2cm}
	
	\begin{definition}\label{d2} Let 
		$G:(\bR^{m},0) \rightarrow (\bR^{p}, 0)$, $m > p > 0$, be a non-constant map germ, and let $V_G = G^{-1}(0)$. If for any $\e > 0$ sufficiently small, there exists a neighborhood $N_\e$ of $\rho^{-1}(\e) \cap V_G$ in $\rho^{-1}(\e)$ such that for any $x \in N_\e \setminus V_G$, we have $\rho^{-1}(\e)  \pitchfork_x G^{-1}(G(x))$, then we say that $G$ is $\rho$-regular.
	\end{definition}
	
	\vspace{0.2cm}
	
	Let us emphasize that $\rho$-regularity depends directly on the choice of the function $\rho$. Recently, the authors of \cite{ADRS} constructed a class of examples where it is possible for a map to satisfy $\rho$-regularity concerning one choice of $\rho$ but not for a different choice, see \cite[Observation 4.13]{ADRS}. It is important to highlight that, up to now, this is the first work in the literature that managed to present clear and detailed examples addressing the crucial question of the dependence or relationship between different choices for the function $\rho$. 
	
	\vspace{0.2cm}
	
	For example, if we consider $\rho: U \to \bR_{\geq 0}$ to be the Euclidean distance to the origin $\rho(x) = \|x\|^2$, which clearly defines the origin, one has that $\rho^{-1}(\e) = S_\e^{m-1}$,  for any $\e > 0$.

	\vspace{0.2cm}
	
	By providing concrete examples, the paper \cite{ADRS} offers a new perspective on this interaction, enabling researchers to better explore this connection and its implications in different contexts. For instance, consider the following question:
	
	\begin{question}\label{q3}
		After considering the Euclidean distance function $\rho(x)=\|x\|^{2}$, what adaptations are necessary for all the results stated with this specific choice to extend to any function $\rho$ as defined above?
	\end{question}

	\vspace{0.2cm}
	
	Next, we state a result regarding the existence of a Milnor tube fibration using $\rho$-regularity. In light of Question \eqref{q3}, from now on, we will consider only the Euclidean distance function $\rho(x)=\|x\|^{2}$.
	
	\vspace{0.2cm}
	
	\begin{theorem}\label{ttf}
		Let $G:(\bR^m,0) \to (\bR^p,0),$ with $m\ge p\ge1$, be a non-constant  analytic map germ, such that $\Disc G = \{0\}$. If $G$ is $\rho$-regular, then it admits a Milnor tube fibration \eqref{tmil}.
	\end{theorem}
	
	\begin{proof} Considering the critical points of the map:
		\begin{equation}\label{e11}
			G_{|}: S_{\e}^{m-1}\cap G^{-1}(B^{p}_{\eta} \setminus \{0\}) \to B^{p}_{\eta} \setminus \{0\},
		\end{equation}
		we see that these points are precisely the points $x \in S_{\e}^{m-1}\cap G^{-1}(B^{p}_{\eta} \setminus \{0\}) $ for which $ S_{\e}^{m-1} \not \pitchfork_x G^{-1}(G(x)) $. From the $\rho$-regularity of the map, it follows that for $\e>0$ and $\eta>0$ sufficiently small, \eqref{e11} is a smooth submersion.
		
		Given the hypothesis $\Disc G = \{0\}$, the map $ G_{|}: B_{\e}^{m}\cap G^{-1}(B^{p}_{\eta} \setminus \{0\} ) \to B^{p}_{\eta} \setminus \{0\}  $ is also a submersion and is surjective. Hence, we have that:
		\begin{equation}\label{f2eq9}
			G_{|}:\overline{B}^m_\e \cap G^{-1}(B^{p}_\eta \setminus \{0\} ) \to B^{p}_\eta \setminus \{0\} 
		\end{equation}
		is a smooth submersion over a manifold with boundary. We claim that \eqref{f2eq9} is proper. Indeed, consider $K \subset {B}^{p}_\eta \setminus \{0\} $ as a compact set. Since $G^{-1}(K) \subset G^{-1}(B^{p}_\eta \setminus \{0\})$, then $\overline{B}^m_\e\cap G^{-1}({B}^{p}_{\eta} \setminus \{0\} )\cap G^{-1}(K)=\overline{B}_{\e}^{m} \cap G^{-1}(\overline{B}^p_\eta)\cap G^{-1}(K)$. Furthermore, the map $G_{|}:\overline{B}^m_\e \cap G^{-1}(\overline{B}^p_\eta) \to \overline{B}^p_\eta$ is proper since it is continuous over a compact domain and $K \subset \overline{B}^{p}_\eta$. Thus, $\overline{B}^m_\e\cap G^{-1}({B}^{p}_{\eta} \setminus \{0\} )\cap G^{-1}(K)$ is compact. We can now apply Ehresmann's Theorem to ensure the existence of the Milnor tube fibration.
		
	\end{proof}

	\section{Thom regularity}

	In this section, we present the definition of Thom regularity and some related results.

	\vspace{0.3cm}

	Let $G: (\mathbb{R}^{m},0) \to (\mathbb{R}^{p},0)$ be a germ of an analytic map, and let $\mathbb{W}$ be a Whitney stratification of $G$. Let $W_\alpha$ and $W_\beta$ be strata of $\mathbb{W}$ such that $W_\alpha \subset \overline{W}_\beta$, and the restrictions $G_{|_{W_{\alpha}}}$ e $G_{|_{W_{\beta}}}$ have constant ranks. Let $x \in W_{\alpha}$. The following definitions can be found in \cite{GLPW} and \cite{JM}.
	
	\vspace{0.3cm}
	
	\begin{definition}\label{defGL2}
		We say that $W_{\beta}$ is \textit{Thom regular over $W_{\alpha}$ at $x$ with respect to $G$}, or equivalently, \textit{the pair $(W_\beta, W_\alpha)$ satisfies the condition of \textbf{Thom $(\aa_G)$-regularity} at $x$}, if the following condition holds:
		
		\vspace{0.3cm}
		
		\begin{itemize}
			\item [($\ast$)] Let $\{x_n\} \subset W_{\beta}$ be a sequence converging to $x$. If $\ker \dd_{x_n}(G_{|_{W_{\beta}}})$ converges to a space $\mathcal{T}$ in an appropriate Grassmann bundle, then $\ker \dd_{x}(G_{|_{W_{\alpha}}}) \subset \mathcal{T}$.
		\end{itemize}
		
		\vspace{0.3cm}
		
		We say that $W_\beta$ is \textit{Thom regular over $W_\alpha$ with respect to $G$}, or equivalently, \textit{the pair $(W_\beta, W_\alpha)$ satisfies the condition of \textbf{Thom $(\aa_G)$-regularity}}, when condition ($\ast$) is satisfied for any point $x\in W_\alpha$.
	\end{definition}

	\vspace{0.3cm}

	\begin{definition}\label{d:thom-general}
		Let $G: (\mathbb{R}^m, 0) \to (\mathbb{R}^p, 0)$ be a germ of an analytic map. We say that there exists a \textbf{Thom $(\aa_G)$-regular stratification in $V_G$}, or that \textbf{$G$ has Thom regularity}\footnote{Or simply, $G$ is Thom regular}, if there exists a regular stratification $(\mathbb{W}, \mathbb{S})$ of a germ $G$, such that $0$ is a point in the stratum $\mathbb{S}$, $V_G$ is the union of strata in $\mathbb{W}$, and the Thom $(\aa_G)$-regularity condition is satisfied at any stratum of $V_G$.
	\end{definition}
	
	\vspace{0.3cm}
	
	\begin{remark}
		Let $G: (\mathbb{R}^m,0)\to (\mathbb{R}^p,0)$ be a germ of an analytic map. In general, there can be two situations regarding the singular set:
		\begin{enumerate}
			\item[(1)] $\Disc G = \{0\}$ as a germ of a set;
			\item[(2)] $\textrm{dim}\, \Disc G > 0$ as a germ of a set.
		\end{enumerate}
		
		\vspace{0.3cm}
		
		\textbf{Suppose case (1) occurs}: If there exists a stratification $\mathbb{W}:=\{W_\alpha\}_{\alpha \in A}$ of $V_G$, such that in a neighborhood $B_\epsilon^m$ of the origin, $\mathbb{W}':=\{B_\epsilon^m \setminus V_G\}\cup \{W_\alpha \cap B_\epsilon^m \}_{\alpha \in A}$ is a Whitney stratification and the pair $(B_\epsilon^m \setminus V_G, W_\alpha)$ satisfies the Thom $(\aa_G)$-regularity condition for any stratum $W_\alpha$, we say that \textit{\textbf{$G$ has Thom regularity in the classical sense}}\footnote{Or simply, $G$ is Thom regular in the classical sense.} or that there exists a \textit{\textbf{Thom $(\aa_G)$-regular stratification in $V_G$ in the classical sense}}, whenever it is necessary to emphasize the fact that $G$ has an isolated critical value.

		\vspace{0.3cm}

		\textbf{Suppose case (2) occurs}: If there exists a stratification $\mathbb{W}:=\{W_\alpha\}_{\alpha \in A}$ of $V_G$, such that in a neighborhood $B_\epsilon^m$ of the origin, $\mathbb{W}':=\{B_\epsilon^m \setminus G^{-1}(\Disc G) \}\cup \{W_\alpha \cap B_\epsilon^m \}_{\alpha \in A}$ is a Whitney stratification and the pair $(B_\epsilon^m \setminus G^{-1}(\Disc G), W_\alpha)$ satisfies the Thom $(\aa_G)$-regularity condition for any stratum $W_\alpha$, we say that \textit{\textbf{$G$ has $\partial$-Thom regularity}}\footnote{Note that $\partial$-Thom regularity reduces to Thom regularity in the classical sense when $\Disc G = \{0\}$.} or that there exists a \textit{\textbf{$\partial$-Thom $(\aa_G)$-regular stratification in $V_G$}}, whenever it is necessary to emphasize the fact that $\textrm{dim}\,\Disc G > 0$.
		
	\end{remark}

	\vspace{0.3cm}

	This weaker condition of $\partial$-Thom regularity was introduced in \cite[Def. 2.1]{Ti1} and \cite[A.1.1.1]{Ti3} for the case of neighborhoods at infinity and adapted in \cite{Ri} for the local case of germs of analytic maps. It was further presented in the article \cite{ART}, where the authors demonstrated that such a condition guarantees the existence of stratified fibration structures.
	The main difference between $\partial$-Thom regularity and Thom regularity in $V_G$ is that in the former, we \textbf{do not} require the Thom $(\aa_G)$-regularity condition to hold for all pairs of strata $(W_\beta, W_\alpha)$ where $W_\beta$ is outside of $V_G$ and $W_\alpha$ is contained in $V_G$. It is only necessary for the condition to be satisfied for pairs where $W_\beta = B_\epsilon^m \setminus G^{-1}(\Disc G)$.
	
	\vspace{0.3cm}
	
	Obviously, the existence of a Thom $(\aa_G)$-regular stratification in $V_G$ implies the existence of a $\partial$-Thom $(\aa_G)$-regular stratification in $V_G$. In this work, we will focus only on the case (1), where $\Disc G = \{0\}$.

	\vspace{0.3cm}

	\begin{example}\cite[Example 5.3]{ACT}
		Consider the map $G(x, y,z) = (y^4-z^2x^2-x^4,xy)$ in three real variables. We have $V_G = \{x=y=0\}=\Sing G$.
		Therefore, $\Disc G = \{0\}$.   The authors at \cite{ACT} have proved that $G$ has Thom regularity in the classical sense.
	\end{example}

	The first result in this section shows that Thom regularity at $V_G$ implies the existence of a neighborhood of $S^{m-1}_{\e} \cap V_G$ in $S^{m-1}_{\e}$ where the fibers of $G$ intersect $S^{m-1}_{\e}$ transversely. In other words, Thom regularity implies $\rho$-regularity and by Theorem \ref{ttf} in the existence of Milnor tube fibration.
	
	\vspace{0.2cm}
	
	\begin{proposition}\label{pre-p1}
		Let $G:(\mathbb{R}^m,0)\to (\mathbb{R}^p,0)$ be a germ of analytic map with isolated critical value. If $ G $ is Thom regular at $V_G$, then for any $\e>0$ sufficiently small, there exists a neighborhood  $N_\e$ of $S^{m-1}_{\e} \cap  V_G$ in $S^{m-1}_{\e}$ such that for any $x \in N_\e \setminus V_G$, we have $S^{m-1}_{\e} \pitchfork_x G^{-1}(G(x))$. 
	\end{proposition}
	
	\begin{proof}
		Since $G$ is Thom regular at $V_G$, there exists a Whitney stratification $\mathbb{S}=\{S_\alpha\}$ of $V_G$ such that for any stratum $S_\alpha$, the pair $(B^{m}_{\e}\setminus V_G, S_\alpha)$ satisfies the Thom $\aa_G$-condition. Moreover, we can consider a stratification such that for any $\e>0$ sufficiently small, the spheres $S^{m-1}_{\e}$ intersect any stratum $S_\alpha$ transversely.
		
		Let $p\in  S^{m-1}_{\e} \cap S_\alpha$. If the statement is not true, then it is possible to obtain a sequence of points $\{x_n\}$ with  $x_n \in S^{m-1}_{\e} \setminus V_G$ such that  $x_n \to p$ and $S^{m-1}_{\e}  \not\pitchfork_{x_n}  G^{-1}(G(x_n))$. Now, considering the limit $\textrm{T}:=\lim \,\textrm{T}_{x_n}G^{-1}(G(x_n))$ in some appropriate Grassmann bundle, by Thom regularity, we have $\textrm{T}_pS_{\alpha} \subset \textrm{T}$, which leads to a contradiction since $\textrm{T}_pS_{\alpha} + \textrm{T}_pS^{m-1}_{\e} = \bR^{m}$. 
		
	\end{proof}
	
	\vspace{0.2cm}
	
	\begin{remark}\label{obs1}
		Since $S^{m-1}_{\e} \cap  V_G \neq \emptyset$, Thom regularity at $V_G$ implies the existence of $\eta >0$ with $0<\eta \ll \e$ such that  $S^{m-1}_{\e} \cap G^{-1}(B^{p}_{\eta}\setminus \{0\}) \subset N_\e$. In other words, the map:
		\begin{equation}\label{rest1}
			G_{|}: S_{\e}^{m-1}\cap G^{-1}(B^{p}_{\eta} \setminus \{0\}) \to B^{p}_{\eta} \setminus \{0\}
		\end{equation}
		is a smooth submersion. This information will be useful to guarantee that Thom regularity at $V_G$ implies the condition of regularity known as the \textit{Milnor Condition (b)}.
	\end{remark}

	\vspace{0.3cm}

	\begin{remark}
		An important fact is that the Thom $(\aa_G)$-regularity condition needs to be verified only at points $p\in \Sing G \cap V_G$. In fact, if $p\in V_G\setminus \Sing G$, then $p$ is a regular point of $G$. Consequently, there exists a neighborhood of $p$ in $V_G$ where all points are regular, and thus it is possible to construct a stratification around $p$ satisfying the Thom $(\aa_G)$-regularity condition.
	\end{remark}
	
	\vspace{0.3cm}

	Initially, the existence of a Thom regular stratification was characterized in terms of polar invariants in \cite{HMS}. In \cite{Ga}, Gaffney used integral closures of modules to study Thom $(\aa_f)$-regularity. In \cite{Hi}, H. Hironaka established Thom regularity for holomorphic functions as follows:
	
	\vspace{0.2cm}
	
	\begin{theorem}\cite{Hi}\label{theorem-hi}
		Let $f: E \to C$ be a germ of an analytic map of a complex algebraic set $E$ over a non-singular complex curve $C$. Then $f$ is Thom regular.
	\end{theorem}
	
	\vspace{0.2cm}
	
	In the case of hypersurfaces in $\bC^{n+1}$, L\^{e} D\~ung Tr\'ang and H. Hamm proved\footnote{See \cite[Theorem 1.2.1]{HL}} that this result still holds and is implied by the classical \L ojasiewicz inequality, which we will recall below:
	
	\vspace{0.2cm}
	
	\begin{definition}\cite[\L ojasiewicz Inequality]{LZ}\label{pre-lojine}
		Let $f:U\to \bC$ be a holomorphic function, where $U\subset \bC^n$ is an open set such that $0\in U$ and $f(0)=0$. We say that $f$ satisfies the \textit{\L ojasiewicz inequality} at the origin if there exists a neighborhood $W$ of $0$ in $U$ such that for any $z\in W$, we have \[c\,\|f(z)\|^{\theta} \le \|\,\nabla f (z)\| \] 
		for some $ \theta \in (0,1)  $ and some $ c >0 $. 
	\end{definition}
	
	\vspace{0.2cm}
	
	\begin{theorem}\label{loj}\cite{BM,Lo1}
		Every holomorphic function satisfies the \L ojasiewicz inequality.
	\end{theorem}
	
	\vspace{0.2cm}
	
	\begin{theorem}\cite{HL,O}\label{pre-r1}
		Every holomorphic function is Thom regular.
	\end{theorem}
	\begin{proof}
		Let $f: U \subset \C^n \to \C$ be a holomorphic function. For a given natural number $N$, define the analytic function $g:U \times \C \to \C$ by setting $g(x,t) = f(x) - t^N$, and consider the analytic set $$A=\left\lbrace (x,t) \in U \times \C; f(x)-t^N=0 \right\rbrace. $$ In what follows, we will find the singular set of the analytic variety $A$. Observe that $\nabla g(x,t) = (\nabla f(x), -Nt^{N-1})$. Thus, $\nabla g(x,t)=(0,0)$ if and only if $\nabla f(x)=0$ and $t=0$. Therefore, the set $\Sing g = \Sing f \times \{0\}$, which clearly satisfies the inclusion $\Sing g \subset V_f \times \{0\}$. Now, consider a Whitney stratification $\bW$ of $A$ such that $A \setminus (V_f \times \{0\})$ and $V_f \times \{0\}$ are unions of strata. This stratification induces a Whitney stratification $\mathcal{S}_N$ of $V_f$.
		
		\vspace{0.2cm}
		
		\noindent
		We claim that for sufficiently large $N$, $\mathcal{S}_N$ satisfies the Thom regularity condition. Indeed, consider $x_n:=(u_n,t_n)$ as a sequence in $A \setminus (V_f \times \{0\}) \cap B$, where $B$ is some compact ball centered at the origin in $U \times \C$, with $x_n \to (x,0) \in (V_f\times \{0\})\cap B$. Assume that the sequence of tangent planes $\textrm{T}\,\left( f^{-1}(f(u_n)\times t_n)\right)$ converges to some limit $\textrm{T}$. Since $x_n \in A \setminus \Sing g$, we know that the tangent plane $\textrm{T}_{x_n}\,A$ is well-defined, and for each fixed $t_n$, the analytic set $f^{-1}(f(u_n)) \times \{t_n\}:=\left\lbrace (x_n, t_n); f(x)=t^N_n\right\rbrace$ is an analytic submanifold of $A$. Hence, we have that $\textrm{T}\,\left( f^{-1}(f(u_n)\times t_n)\right)$ is a complex vector subspace of codimension $1$ in $\textrm{T}_{x_n}\,A$. Assume that 
		\begin{equation}\label{e9}
			\textrm{T}_xV_x \not \subset \textrm{T},
		\end{equation}
		where $V_x$ is some stratum of $V_f$ containing the point $x$. After possibly passing to a subsequence, we can assume that $\textrm{T}_{x_n}\, A$ converges to $\tau$. As $\bW$ satisfies the Whitney condition (a), we have $\textrm{T}_xV_x \subset \tau$. From the inclusion $\textrm{T}\,\left( f^{-1}(f(u_n)\times t_n)\right) \subset \textrm{T}_{x_n}\,A$, it follows that $\textrm{T} \subset \tau$. Note that the complex dimension of $\tau$ is $n$ and that of $\textrm{T}$ is $n-1$, therefore, from \eqref{e9} and \eqref{e91}, we get
		\begin{equation}\label{e91}
			\tau = \textrm{T} \oplus \textrm{T}_xV_x.
		\end{equation}
		On the other hand, $\textrm{T} \subset\C^n \times \{0\}$ and $\textrm{T}_xV_x \subset \C^n \times \{0\}$. Thus, from the equality in \eqref{e91}, we have
		\begin{equation}\label{e92}
			\tau = \C^n \times \{0\}.
		\end{equation}
		Now, to analyze the limit of the sequence of tangent planes $\textrm{T}_{x_n}\, A$, let us consider the vector $$\frac{\nabla g(x_n)}{\|\nabla g(x_n)\|}=\frac{1}{\|\nabla f(u_n)\|\sqrt{1+|Nt^{N-1}_n|^2/\|\nabla f(u_n)\|^2}}\left( \nabla f(u_n),-Nt^{N-1}_n\right)  .$$ Since $x_n=(u_n,t_n) \in A$, we have $f(u_n)=t_n^N$. Thus,  $|t_n|^{N-1} =  |f(u_n)|^{ \frac{N-1}{N}} $. Therefore, we have the equality $$ |Nt^{N-1}_n|/\|\nabla f(u_n)\| =N|t_n|^{N-1}/\|\nabla f(u_n)\| = N|f(u_n)|^{\frac{N-1}{N}}/\|\nabla f(u_n)\|. $$
		
		\vspace{0.2cm}

		\noindent
		Now, it follows from the \L ojasiewicz inequality that $ \|\nabla f(u_n)\| \ge c|f(u_n)|^\theta $, thus, $$ |Nt^{N-1}_n|/\|\nabla f(u_n)\| \le |f(u_n)|^{\frac{N-1}{N}-\theta}.$$ Taking $N$ such that $N > \frac{1}{1-\theta}$, we obtain that $\frac{N-1}{N} - \theta > 0$. As $x_n=(u_n,t_n) \to (x,0)\in V_f \times \{0\}$, we have $f(u_n) \to f(x)=0$. Therefore, $|Nt^{N-1}_n|/\|\nabla f(u_n)\|  \to 0$ as $n \to +\infty$, and $\nabla g(x_n) / \|\nabla g(x_n)\| \to (z_0,0)$, where $z_0 \in \C^n \setminus \{0\}$, which contradicts the fact that $(z_0,0)$ is a normal vector to $\tau$. Thus, we have $\textrm{T}_xV_x  \subset \textrm{T}$, i.e., $\bW$ satisfies the Thom regularity condition, and in particular, $\mathcal{S}_N$ is a Thom $ (a_f) $-regular stratification of $V_f$.
	\end{proof}
	
	\vspace{0.3cm}
	
	Unfortunately, for a real analytic map germ $G: (\mathbb{R}^m,0) \to (\mathbb{R}^p,0)$, Thom regularity in $V_G$ is not generally satisfied, this is demonstrated by the following example.

	\vspace{0.3cm}
	
	\begin{example}\cite[Example 1.4.9]{Ha}\label{eNT1}
		Consider the map $G(x, y,z) = (x, y(x^2 + y^2) + xz^2)$ in three real variables. We have $V_G = \{x=y=0\}=\Sing G$.
			Therefore, $\Disc G = \{0\}$. We claim that $G$ is not Thom regular in $V_G$ in the classical sense. 
			
			Indeed, let $\mathbb{W}=\{W_i\}$ be a Whitney stratification of $V_G$, and consider the point $p=(0,0,z)\in W_i$ for some stratum of positive dimension $W_i$. In this case, we have $\textrm{T}_{p}W_i = \textrm{span}{(0,0,1)}$. Note that the sequence of points $p_n = \left(\frac{1}{n},0,z\right)\in B^3_{\epsilon} \setminus V_G$ converges to $p$, and for each natural number $n$, $$\ker \dd_{p_n}(G_{|_{B^3_\epsilon \setminus V_G}}) =\textrm{T}_{p_n}G^{-1}(G(p_n)) = \textrm{span} \left\{v_n\right\},$$ where $$v_n = \left(0, \frac{-2z}{\sqrt{4z^2+ \frac{1}{n^2}}}, \frac{1}{\sqrt{4z^2n^{2}+1}}\right).$$ 
			
			Since $\lim_{n}v_n = (0,\pm1,0)$, where the plus or minus sign depends on the sign of $z$, we have $\lim_{n}(\ker \dd_{p_n}(G_{|{B^3\epsilon \setminus V_G}})) =\textrm{span}{(0,1,0)}=:\cT$.
			That is, $\textrm{T}_{p}W_i \not \subset \cT$, meaning that $\left(B^3_{\epsilon} \setminus V_G, W_i \right)$ does \textbf{not} satisfy the Thom $(\aa_G)$-regularity condition for any stratum $W_i$ of $V_G$. Therefore, $G$ is not Thom regular in $V_G$ in the classical sense.
		\end{example}
		
		\vspace{0.3cm}
		
		Deciding whether a real analytic map germ is Thom regular can be a very challenging task, except possibly in the case of germs with an isolated singularity at the origin, where Thom regularity naturally holds. This is stated in the following theorem.
		
		\vspace{0.3cm}

		\begin{theorem}\label{ttr}
			If $f: \R^{n} \to \R^{m}$ is a real analytic function with an isolated singularity at the origin, then $f$ is Thom regular in $V_f$.
		\end{theorem}
		\begin{proof}
			Choose any stratification of $V(f)$, with $S_{0}={0}$ being one of its strata. If a sequence of points $x_{k}$ converges to $0$ with $T_{x_{k}}f^{-1}(f(x_{k}))$ converging to $T$, then we have $T_{0}S_{0} \in T$, so $T_{0}S_{0}$ will be a point. Assume, therefore, that $p$ is a point in $V(f)$ with $p \neq 0$.
			By the implicit function theorem, we can find local coordinates $\phi$ and $\psi$ such that $\tilde{f}=\psi \circ f \circ \phi^{-1}$ is the projection given by $\tilde{f}(x_{1}, \ldots, x_{n})=(x_{1}, \ldots, x_{n-m}, 0, \ldots, 0)$. For any point $q$ in a neighborhood of $p$, let $\tilde{q}=\phi(q)$. It follows from the definition of $\psi$ that $\tilde{p}=0$. Let $p_{k}$ be a sequence of points converging to $p$. Write $\tilde{p}_{k}$ as $\tilde{p}_{k}=(\tilde{p}_{k,1}, \ldots, \tilde{p}_{k,n})$. Then, for $k$ sufficiently large, we have $T_{\tilde{p}_{k}}\tilde{f}^{-1}(\tilde{f}(\tilde{p}_{k}))$ given by $(\tilde{p}_{k,1}, \ldots, \tilde{p}_{k,n-m}) \times \R^{m} \subset \R^{n}$.
			As $p_{k}$ tends to $p$, $T_{\tilde{p}_{k}}\tilde{f}^{-1}(\tilde{f}(\tilde{p}_{k}))$ must tend to
			$(0,\ldots,0,0) \times \R^{m}=T_{\tilde{p}}\phi(V(f))$. But since $\phi$ is a diffeomorphism, we then have $T_{p}(V(F))=T$, where $T$ is the limit of $T_{p_{n}}f^{-1}(f(p_{n}))$.
		\end{proof}
		
		\vspace{0.2cm}
		
		The next two examples have an isolated singularity at the origin; therefore, from Theorem \ref{ttr}, they are Thom regular at $V_G$.
		
		\vspace{0.2cm}
		
		\begin{example}\label{ex2}
			The map $ G(x_1,x_2,x_3,x_4)=(x_1,3x_1^2x_2+x_2^3+x_3^2+x_4^2 ) $ has an isolated singularity at the origin.
		\end{example}
		
		\begin{example}
			The map $G:=(G_1,G_2): (\bR^3,0)\to (\bR^2,0)$ given by  $G_1(x_1,x_2,x_3)=x_2$ and $G_2(x_1,x_2,x_3)=x_2^2 + x_1(x_1^2 + x_2^2 + x_3^2)$ has an isolated singularity at the origin; therefore, it is Thom regular at $V_G$. 
		\end{example}
		
		\vspace{0.3cm}
		
		In recent years, researchers have been exploring conditions that lead to Thom regularity for specific classes of real analytic map germs with isolated critical values at the origin. These conditions are of significant interest due to their direct connection with the existence of locally trivial smooth fibrations. Several noteworthy results have been obtained in different contexts.
		
		\vspace{0.3cm}
		
		Massey, in \cite{Ma}, formulated a sufficient condition of the \L ojasiewicz type. Oka, in \cite{Oka3} considered mixed functions $G: \mathbb{C}^n \to \mathbb{C}$ that are strongly Newton non-degenerate. In the case of certain functions of the form $f\bar{g}$, where $f$ and $g$ are germs of holomorphic functions, in 2 variables was addressed in \cite{PS3}, in 3 variables in \cite{FM}, and in the general case of $n$ variables in \cite{PT}.
		
		\vspace{0.3cm}
		
		Furthermore, for functions of the $MSL$ type $G: \mathbb{C}^n \to \mathbb{C}$, efficient and more easily verifiable conditions were discovered by the authors in \cite{R}. These conditions have been instrumental in generating new examples of Thom regular real analytic maps in the classical sense. The combined efforts in exploring these conditions have significantly enriched our understanding of Thom regularity and its implications in different settings.



		\vspace{0.3cm}
		
		\section{New examples of maps with Thom regularity}
		
		In what follows, we will consider the following definition:
		
		\begin{definition}
			We say that a germ map between euclidean spaces $G: (\mathbb{R}^m, 0) \to (\mathbb{R}^p, 0)$ is a \textit{horizontally weakly conformal maps} if $  \left\langle \nabla G_i, \nabla G_j \right\rangle =0  $ for $ i\neq j $, $ i,j \in{1,\dots, p} $, and $\|\nabla G_i \|^2 = \|\nabla G_j \|^2$ for all $ i,j \in{1,\dots, p} $.
		\end{definition}
		
		\vspace{0.2cm}
		
		\begin{example}
			Consider the map $ G:=(G_1,G_2):(\mathbb{R}^4,0) \to (\mathbb{R}^2,0) $ given by 	
			$$ \left\lbrace \begin{array}{cl}
				G_1(x,y,z,w)&= -w^2x^3+3w^2xy^2+6wx^2yz-2wy^3z+x^3z^2-3xy^2z^2    \medskip\\
				G_2(x,y,z,w)&= 3w^2x^2y-w^2y^3+2wx^3z-6wxy^2z-3x^2yz^2+y^3z^2   , \medskip\\

			\end{array}\right. $$ 	
			
			We have
			
			$$ \begin{array}{cl}
				\nabla G_1  &=  (-3w^2x^2+3w^2y^2+12wxyz+3x^2z^2-3y^2z^2)\vec{e}_1\\
				&+ (6w^2xy+6wx^2z-6wy^2z-6xyz^2) \vec{e}_2\\
				&+  (6wx^2y-2wy^3+2x^3z-6xy^2z)\vec{e}_3\\
				&+  (-2wx^3+6wxy^2+6x^2yz-2y^3z) \vec{e}_4\\
			\end{array} $$
			
			and
			
			$$ \begin{array}{cl}
				\nabla G_2   &= (6w^2xy+6wx^2z-6wy^2z-6xyz^2) \vec{e}_1\\
				&+ (3w^2x^2-3w^2y^2-12wxyz-3x^2z^2+3y^2z^2) \vec{e}_2\\
				&+  (2wx^3-6wxy^2-6x^2yz+2y^3z)\vec{e}_3\\
				&+  (6wx^2y-2wy^3+2x^3z-6xy^2z) \vec{e}_4\\
			\end{array} $$
			
			\vspace{0.2cm}

			\noindent
			Portanto, 	$ \left\langle \nabla G_1, \nabla G_2 \right\rangle =0 $, 	e $ \|\nabla G_1 \|^2 = (x^2+y^2)^2(w^2+z^2)(9w^2+4x^2+4y^2+9z^2) = \|\nabla G_2 \|^2  $. Hence, $ G $ is a horizontally weakly conformal map.
		\end{example}
		
		\vspace{0.2cm}

		\begin{remark}
			If $G:(\mathbb{R}^{m},0)\to (\mathbb{R}^{p},0)$ is a germ of a horizontally weakly conformal map, then $\Sing G = \left\lbrace x\in \mathbb{R}^m \mid JG(x) = 0 \right\rbrace$. In fact, since  $$ \Sing G = \left\lbrace x\in \bR^m\,|\,\textrm{det}\left(  JG(x) \dot JG(x)^{T} \right) = 0 \right\rbrace$$ and $G$ is a horizontally weakly conformal map, we have that $\textrm{det}\left( JG(x) \cdot JG(x)^{T} \right) = |\nabla G_j(x)|^{2p}$ for an arbitrarily chosen $j \in {1,\ldots,p}$. Therefore, if $x\in \Sing G$, then $\nabla G_j(x) = 0$ for all $j=1,\ldots,p$, which means $JG(x) = 0$.
		\end{remark}

		\vspace{0.2cm}
		
		The next result guarantees that horizontally weakly conformal maps satisfy Thom regularity. A proof can be found in \cite{Ma} or \cite{ADRS}, for example.
		
		\vspace{0.2cm}
		
		\begin{theorem}\label{tma}\cite{ADRS,Ma}
			Let $G: (\mathbb{R}^m, 0) \to (\mathbb{R}^p, 0)$ be an analytic horizontally weakly conformal map germ. Then $ \Disc(G) = \{0\} $ and $ G $ is Thom regular at $ V_G $.
		\end{theorem}
		
		\vspace{0.2cm}
		
		Let us point out that D. Massey in \cite{Ma}, have called of {\it simple \l - map}  a map  $ G:=(G_1(x), G_2(x), \ldots, G_p(x))$ from an open subset $ U $ of $ \bR^m $ to $ \bR^p $ such that the gradients $ \nabla G_1(x),\nabla G_2(x), \ldots,\nabla G_p(x) $ are always pairwise orthogonal and have the same length, see \cite[Definition 3.5]{Ma}. Therefore, for a map between Euclidean spaces, horizontally weakly conformal maps and simple \l -maps correspond to the same kind of maps. 
		
		\vspace{0.2cm}
		
		Moreover, D. Massey also showed that if $ G  $ is a simple \l -map which is real analytic, it satisfies the {\it strong \L ojasiewicz inequality} at each point $ x\in U $, see \cite[Corollary 3.2 and Remark 3.4]{Ma}. 
		
		\vspace{0.2cm}
		
		\begin{example}\label{e21}
			Consider the map $ G:=(G_1,G_2):(\mathbb{R}^8,0) \to (\mathbb{R}^2,0) $ given by  
			
			$$ \left\lbrace \begin{array}{ccl}
				G_1(x,y,z,w,a,b,c,d)&= -w^2x^2+w^2y^2+4wxyz+x^2z^2-y^2z^2+ac+bd&   \medskip\\
				G_2(x,y,z,w,a,b,c,d)&=-2w^2xy-2wx^2z+2wy^2z+2xyz^2-ad+bc,&  \medskip\\

			\end{array}\right. $$ 
			
			We have
			
			$$ \begin{array}{cl}
				\nabla G_1  &= (-2w^2x+4wyz+2xz^2) \vec{e}_1\\
				&+ (2w^2y+4wxz-2yz^2) \vec{e}_2\\
				&+  (4wxy+2x^2z-2y^2z)\vec{e}_3\\
				&+ ( -2wx^2+2wy^2+4xyz) \vec{e}_4\\
				&+  c \vec{e}_5\\
				&+   d \vec{e}_6 \\
				&+   a\vec{e}_7\\
				&+   b \vec{e}_8
			\end{array} $$
			
			and 
			
			$$ \begin{array}{cl}
				\nabla G_2   &= (-2w^2y-4wxz+2yz^2) \vec{e}_1\\
				&+  (-2w^2x+4wyz+2xz^2)\vec{e}_2\\
				&+  (-2wx^2+2wy^2+4xyz)\vec{e}_3\\
				&+   (-4wxy-2x^2z+2y^2z)\vec{e}_4\\
				&+  -d \vec{e}_5\\
				&+   c \vec{e}_6 \\
				&+   b\vec{e}_7\\
				&+    -a\vec{e}_8
			\end{array} $$

			Therefore, 	
			
			\[ \left\langle \nabla G_1, \nabla G_2 \right\rangle =0 \]
			
			$$ \begin{array}{lll}
				\|\nabla G_1 \|^2 &=4w^4x^2+4w^4y^2+4w^2x^4+8w^2x^2y^2+8w^2x^2z^2+4w^2y^4+8w^2y^2z^2\medskip\\	
				& = +4x^4z^2+8x^2y^2z^2 +4 x^2 z^4+4 y^4 z^2+4 y^2 z^4+a^2+b^2+c^2+d^2 \medskip\\
				& = \|\nabla G_2 \|^2 \medskip\\
			\end{array}$$
			
			Hence, Theorem \ref{tma} implies that $ G $ is Thom regular in $ V_G $, with $ \Disc G = \{0\} $.
			
		\end{example}	
		
		\vspace{0.2cm}
		
		The next result uses the complex structure to build a horizontally weakly conformal between Euclidean spaces.
		
		\vspace{0.2cm}
		
		\begin{proposition}\cite{ADRS}\label{t1} Let $ (G_1,G_2,G_3,G_4):(\bR^{2n},0)\to (\bR^4,0 ) $ be a  polynomial map germ such that $ (G_1,G_2) $ and $ (G_3,G_4) $ are horizontally weakly conformal maps. If $ \left\langle \nabla G_1, \nabla G_3 \right\rangle -\left\langle \nabla G_2, \nabla G_4 \right\rangle =0  $ and $ \left\langle \nabla G_1, \nabla G_4 \right\rangle +\left\langle \nabla G_2, \nabla G_3 \right\rangle =0  $, then $ \left( G_1G_3 - G_2G_4, G_1G_4+G_2G_3 \right):(\bR^{2n},0)\to (\bR^2,0 )  $ is a horizontally weakly conformal map.
		\end{proposition}
		
			
			
			
		
		\vspace{0.2cm}
		
		\begin{example}
			Let $ G(x,y,z,w,a,b)=(xz-yw,xw+yz,ax+by,-ay+bx) $. Define $ G_1:= xz-yw$, $ G_2:=xw+yz $, $G_3:=ax+by $ and $ G_4:=-ay+bx $. One can show that $ (G_1,G_2) $ and $ (G_3,G_4) $ are a horizontally weakly conformal map.  However,  $ \left( G_1G_3 - G_2G_4, G_1G_4+G_2G_3 \right):(\bR^{2n},0)\to (\bR^2,0 )  $ is not a a horizontally weakly conformal map. Note that the additional hypothesis of Proposition \ref{t1} about the gradient of the component functions fails.
		\end{example}
		
		\vspace{0.2cm}
		
		\begin{example}
			Let $ G(x,y,z,w,a,b)=(xz-yw,xw+yz,ax+by,ay-bx) $. Define $ G_1:= xz-yw$, $ G_2:=xw+yz $, $G_3:=ax+by $ and $ G_4:=ay-bx $. One can show that $ (G_1,G_2) $ and $ (G_3,G_4) $ are a horizontally weakly conformal map.  Since $ \left\langle \nabla G_1, \nabla G_3 \right\rangle -\left\langle \nabla G_2, \nabla G_4 \right\rangle =0  $ and $ \left\langle \nabla G_1, \nabla G_4 \right\rangle +\left\langle \nabla G_2, \nabla G_3 \right\rangle =0  $, it follows from Proposition \ref{t1} that $ (H_1,H_2):(\bR^{6},0)\to (\bR^2,0 ) $ given by $H_1 (x,y,z,w,a,b) = (xz-yw)(ax+by)-(xw+yz)(ay-bx)$ and $H_2(x,y,z,w,a,b) = (xz-yw)(ay-bx)+(xw+yz)(ax+by)  $ is a horizontally weakly conformal map.
		\end{example}
		
		\vspace{0.2cm}

		
		\begin{definition}
			We say that $f:(\bR^{m}\times \bR^{n},0)\to (\bR^{p},0)$ and $g:(\bR^{m}\times \bR^{n},0)\to (\bR^{k},0)$ are \textbf{separable variable} germs if   \[\dfrac{\partial f_t(x,y)}{\partial y_j} = 0 \textrm{   e   } \dfrac{\partial g_s(x,y)}{\partial x_i} = 0,\] for $ t= 1,\ldots, p $, $ j=1,\ldots, n $, $ s= 1,\ldots, k $, and $ i= 1,\ldots, m $. In this case, we can simply write $ f(x,y):=f(x) $ and $ g(x,y):=g(y) $, with $f:(\bR^{m},0)\to (\bR^{p},0)$ and $g:(\bR^{n},0)\to (\bR^{p},0)$, respectively.
		\end{definition}
		
		\vspace{0.2cm}
		
		Consider the map
		\[G:=f+g:(\bR^{m}\times \bR^{n},0)\to (\bR^{p},0) \]
		where $f:(\bR^{m},0)\to (\bR^{p},0)$ and $g:(\bR^{n},0)\to (\bR^{p},0)$ are germs of maps with separable variables. We have that $V_G= {(x,y)\in\bR^{m}\times \bR^{n},|, f(x)+g(y)=0 }$. Therefore, $ V_f \times V_g \subset V_G $. If $ \Disc f =\{0\} $ and $ \Disc g = \{0\} $, then $ \Sing f \times \Sing g \subset V_G $. Moreover, by definition, $ \nabla G_j(x,y) = (\nabla f_j(x),\nabla g_j(x)) $ for every $ j=1,\ldots,p $. Consequently, $\Sing G \subset \Sing f \times \Sing g$ and $ \Disc G = \{0\} $. In particular, $ V_G \cap \Sing G = \Sing G $.
		
		\vspace{0.2cm}
		
		\begin{proposition}
			If $f:(\bR^{m},0)\to (\bR^{p},0)$ and $g:(\bR^{n},0)\to (\bR^{p},0)$ are horizontally weakly conformal  maps, then the sum  $G:=f+g:(\bR^{m}\times \bR^{n},0)\to (\bR^{p},0)$ is a horizontally weakly conformal map. In particular, $G$ is Thom regular in the classical sense.
		\end{proposition}
		
		\vspace{0.2cm}
		
		\begin{corollary}
			Let $ (G_1, G_2):(\bR^{2m},0)\to (\bR^2,0 ) $ and $ (G_3, G_4):(\bR^{2n},0)\to (\bR^2,0 ) $ be germs of polynomial horizontally weakly conformal map that are in separable variables. Then $ \left( G_1G_3 - G_2G_4, G_1G_4+G_2G_3 \right):(\bR^{2(m+n)},0)\to (\bR^2,0 )  $ a horizontally weakly conformal map. In particular, it is Thom regular in the classical sense.
		\end{corollary}
		
		\vspace{0.2cm}

		\subsection{Mixed functions}
		
		A function $f:\mathbb{C}^{n}\rightarrow \mathbb{C}$  is called a \textit{mixed polynomial function} if  $ f(z)=f\left(\textbf{z}, \bar{\textbf{z}} \right)= \Sigma_{\nu, \mu}c_{\nu, \mu} \textbf{z}^{\nu}\bar{\textbf{z}}^{\mu} $,  where $c_{\nu, \mu} \in \bC$,  $\textbf{z}^{\nu}$  and $\bar{\textbf{z}}^{\mu}$ are multi-monomials in the variables $z_j,\bar{z}_j$, with $j=1,\ldots,n$.  	Mixed singularities have been systematically studied by many researchers in the last years, 
		for instance  \cite{AC,ACT-inf,AR,ART,ART2,FM,BPS,CT0,CT2,CT1,CSS3,EO,IKO,Oka2,Oka3,Oka4,Oka5,Oka6,Oka7,PS0,PS2,PS3,PS1,R,Ri,RSR,RSV,S1,Ti4}.
		
		One may view $f$  as a real analytic map of $2n$ variables $(x,y)$ from $\bR^{2n}$ to $\bR^{2}$ by identifying  $\mathbb{C}^n$ with $\bR^{2n}$, $(z_{1},\ldots, z_{n}) = z\mapsto  (x,y)=(x_1,y_1,\ldots,x_n,y_n)$,  writing $\textbf{z=x+iy} \in \mathbb{C}^{n}$, where $z_{j}=x_{j}+iy_{j} \in \mathbb{C}$ with $x_j, y_j \in \bR$ for $j=1,...,n$. Moreover, one can define the \textit{holomorphic} and \textit{anti-holomorphic} gradients of $ f $ as follows:
		\[\textrm{d} f:=\left(\frac{\partial f}{\partial z_{1}},...,\frac{\partial f}{\partial z_{n}} \right) \textrm{  and  } \bar{\textrm{d}}f:= \left(\frac{\partial f}{\partial  \bar{z}_1},...,\frac{\partial f}{\partial \bar{z}_{n}} \right).\]
		
		\vspace{0.2cm}
		where, 
		\vspace{0.2cm}
		\begin{equation*}
			\frac{\partial f}{\partial z_j}=\frac{\partial u}{\partial z_j}+i \frac{\partial v}{\partial z_j} \textrm{   and   } \frac{\partial f}{\partial \bar{z}_j}=\frac{\partial u}{\partial \bar{z}_j}+i \frac{\partial v}{\partial \bar{z}_j}.
		\end{equation*}
		
		\vspace{0.2cm}
		
		For the next result, whenever $\vec{u}$ and $\vec{v}$ are vectors in $\mathbb{C}^n$ we denote its Hermitian product by $\left\langle\vec{u}, \vec{v} \right\rangle_\mathbb{C}$.

		\vspace{0.2cm}

		\begin{proposition}\cite{ADRS}\label{c1}
			A  mixed polynomial function germ $f: (\mathbb{C}^m,0) \to (\mathbb{C},0) $ is a horizontally weakly conformal maps if and only if $\left\langle \overline{\dd f}, \bar{\dd}f \right\rangle _\mathbb{C}=0$. 
		\end{proposition}

			\begin{proposition}\cite[Corollary 4.4]{ADRS}
				Let $f,g:{\bC}^n \to {\bC}$ be polynomial holomorphic functions. The mixed function $ F=f\bar{g} $ is a horizontally weakly conformal maps if and only if $ \left\langle \overline{\dd f}, \overline{\dd g} \right\rangle _\mathbb{C} =0 $.
			\end{proposition}
			
			The next result was proved in \cite[Proposition 4]{R}, see also \cite[Proposition 4.5]{ADRS}. In our setting, it reads:
			
			
			\begin{proposition}\cite{ADRS,R}\label{alg1}
				Let $ I $ be a non-empty subset of $ \left\lbrace 1,2,\ldots,n \right\rbrace  $ and $ f $ be a mixed polynomial function of $ n $ variables $ \{z_1, z_2, \ldots,z_n\} $ that is holomorphic in the variables $ \left\lbrace z_i\,|\,i\in I \right\rbrace  $ and anti-holomorphic in the complementary variables. Namely, $ \dfrac{\partial f}{\partial \bar{z}_i}=0 $, with $ i\in I $ and $ \dfrac{\partial f}{\partial {z}_j}=0 $, with $ j\not\in I $. Then $ f $ is a horizontally weakly conformal map. Also, the set of this class of horizontally weakly conformal maps (fixing I) is stable under addition and multiplication.
			\end{proposition}

			\vspace{0.2cm}

			\begin{example}\cite{ADRS}
				Consider the mixed functions $ f,g:\bC^{3}\to \bC $ given by $ f(x,y,z)=xy-\bar{z} $ and $ g(x,y,z)=x^2\bar{z} $. It follows from Proposition \ref{alg1} that $ f $, $ g $ and $ F=(xy-\bar{z})x^2\bar{z} $ are horizontally weakly conformal maps.
			\end{example}

			\vspace{0.2cm}
			
			\begin{remark} We can consider the following algorithm given in \cite[Proposition 4]{AR} to construct  classes of mixed polynomial horizontally weakly conformal maps $ f $ as in the Proposition \ref{alg1}:
				\begin{itemize}
					\item [i.] Fix a copy of $\mathbb{C}^n,$ $n\geq 2$, and a coordinate system $z=(z_{1},\ldots, z_ {n})$. 
					\item[ii.] For each $1\le k <n$ choose natural numbers $i_1,\ldots,i_k \in \{1,2,\ldots,n\}$ with $i_1<i_2<\ldots<i_k$ and fix the coordinates $(z_{i_{1}},\ldots,z_{i_{k}})$. For the complementary ordered list $q_{1}, \ldots, q_{n-k}\in \{1,2,\ldots,n\}\m\{i_1,\ldots,i_k \}$, $q_{1}< \ldots< q_{n-k}$, consider the remaining coordinates $(z_{q_{1}},\ldots,z_{q_{n-k}})$.
					\item[iii.] For any natural numbers $j$, $t$ and $p$, choose arbitrary holomorphic functions $f_j(z_{i_{1}},\ldots,z_{i_{k}})$, $r_t(z_{i_{1}},\ldots,z_{i_{k}})$, $g_j(z_{q_{1}},\ldots,z_{q_{n-k}})$ and $h_p(z_{q_{1}},\ldots,z_{q_{n-k}})$.
					\item[iv.] Define the mixed function germ $f:(\mathbb{C}^n,0)\to (\mathbb{C},0)$ by $f(z_1,\ldots,z_n)=$
					$$\sum_{\alpha=1}^{j}f_{\alpha}(z_{i_{1}},\ldots,z_{i_{k}})\overline{g_{\alpha}(z_{q_{1}},\ldots,z_{q_{n-k}})}+\sum_{\beta=1}^{t}r_\beta(z_{i_{1}},\ldots,z_{i_{k}})+\sum_{\gamma=1}^{p}\overline{h_\gamma (z_{q_{1}},\ldots,z_{q_{n-k}}) }.$$
				\end{itemize}
			\end{remark}
			
			\vspace{0.3cm}
			
			\begin{example}
				Let $ f:(\mathbb{C}^5,0) \to (\mathbb{C},0) $ given by $$ f(z_1,z_2,z_3,z_4,z_5)=z_{1}^{4}z_{5}^3\bar{z}_{2}^{5}+z_3^2 \bar{z}_4 +z_{1}^{4}-z_{3}^6 -\bar{z}_2^7\bar{z}_{4}^{3}.$$ Considering $f_1(z_1,z_3,z_5)=z_{1}^{4}z_{5}^3$, $f_2(z_1,z_3)=z_3^2$, $g_1(z_2,z_4)=z_{2}^{5}$, $g_2(z_2,z_4)=z_4$, $r(z_1,z_3)=z_{1}^{4}-z_{3}^6$ and $h(z_2,z_4)= z_2^7z_{4}^{3}$, one has $f=f_1\bar{g}_1+f_2\bar{g}_2 +r-\bar{h}$. By Proposition \ref{alg1}, $f$ is a mixed horizontally weakly conformal map.
			\end{example}
			
			\vspace{0.2cm}

		\section{The Thom Irregular Points set}
		
		In the article \cite{PT}, the authors addressed the problem of the existence of a Thom $(\aa_G)$-regular stratification for germs of functions of the form $f\bar{g}: (\mathbb{C}^n,0) \to (\mathbb{C},0)$ with $n > 1$, where $f$ and $g$ are germs of non-constant holomorphic functions, and the "bar" denotes complex conjugation. They introduced the concept of \textit{Thom irregular points} as follows:
		
		\begin{equation}\label{eq:Thom}
			NT_{f\bar{g}}:=\left\lbrace x \in V_{f\bar{g}} \, | \begin{tabular}{l}
				\textit{there does not exist a Thom stratification Thom  $ (a_{f\bar{g}}) $-regular on $ V_{f\bar{g}} $ }   \\
				\textit{such $ x $ belongs to a stratum of positive dimension}.
			\end{tabular}
			\right\rbrace 
		\end{equation}

		\vspace{0.3cm}
		
		With these notations and definitions, they proved Theorem \ref{tct3}, which allows, under certain conditions, to determine the Thom regularity of functions of the type $ f\bar{g} $ by studying the behavior of maps of the type $ (f,g) $.
		
		\vspace{0.3cm}
		
		\begin{theorem}\cite{PT}\label{tct3}
			Let $ f,g:(\mathbb{C}^n,0) \to (\mathbb{C},0) $ be germs of holomorphic functions such that $ \Disc(f,g) $ contains only curves tangent to the coordinate axes. Then
			\[NT_{f\bar{g}} \subset NT_{(f,g)},\]
			where $ NT_{(f,g)} $ is defined analogously, replacing the function $ f\bar{g} $ with $ (f,g) $ in the above definition. 
			
		\end{theorem}
		
		\vspace{0.3cm}
		
		Let $ G: (\mathbb{R}^m, 0) \to (\mathbb{R}^p, 0) $, with $ m>p\ge 2 $, be a germ of an analytic map such that $ \Disc G = \{0\} $. In this work, we will consider the set $ NT_G $ of {\it Thom irregular points of $ G $} as in \eqref{eq:Thom}, replacing the function $ f\bar{g} $ with $ G $.
		
		\vspace{0.3cm}
		
		It is clear that if there exists a sufficiently small $ \epsilon > 0 $ such that $B_{\e}^{m} \cap NT_G = \emptyset $, then $ G $ is Thom regular in $ V_G $, meaning that there exists a Thom $ (\aa_G) $-regular stratification in $ V_G $.
		
		\vspace{0.3cm}
		
		An immediate consequence of Theorem \ref{tct3} is the following criterion for deciding whether the map $ f\bar{g} $ is Thom regular.
		
		\vspace{0.3cm}
		
		\begin{corollary}
			Let $ f,g:(\mathbb{C}^n,0) \to (\mathbb{C},0) $ be germs of holomorphic functions such that $ \Disc(f,g) $ only contains curves tangent to the coordinate axes. If the map $ (f,g) $ is Thom regular, then $ f\bar{g} $ is Thom regular\footnote{In the classical sense. In fact, it was shown in \cite[Theorem 2.3]{PT} that the hypothesis \textit{``$ \Disc(f,g) $ only contains curves tangent to the coordinate axes''} is equivalent to $ \Disc f\bar{g} = \{0\} $}. In particular, if $ (f,g) $ is an ICIS, then $ f\bar{g} $ is Thom regular.
		\end{corollary}
		
		\vspace{0.3cm}
		
		We can use the set $ NT_G $ to show that a map $ G:=(G_1, \ldots, G_p) $ is not Thom regular in $ V_G $. Note that a point $ p_0$ belongs to the set $ B_{\e}^{m}\cap NT_G$, for sufficiently small $ \e> 0 $, if and only if \textbf{for every} Whitney stratification $ \mathbb{W}_1=\left\lbrace B^{m}_{\e} \m V_G , W_\alpha \right\rbrace_\alpha $ of $ B^{m}_{\e}$, where $ \{W_\alpha \}_\alpha $ forms a Whitney stratification of $ V_G $, and for any positive-dimensional stratum $ W_{p_{0}} $ of $ \mathbb{W}_1 $ with $ p_{0} \in W_{p_{0}}\subset \Sing G $, there exists an analytic curve $ \gamma_{p_{0}} \subset B^{m}_{\e} \m V_{G}$ such that $ \gamma_{p_{0}} \to p_0 $ and \[\dfrac{n_{c(t)}(\gamma_{p_{0}}(t))}{\|n_{c(t)}(\gamma_{p_{0}}(t))\|} \to v_0 \notin [T_{p_{0}} W_{p_{0}}]^{\perp} \] where $n_{c(t)}(\gamma_{p_{0}}(t)) = c_1(t)\nabla G_1(\gamma_{p_{0}}(t))+ \cdots +c_p(t)\nabla G_p(\gamma_{p_{0}}(t))$ is a normal vector to the fiber of $G$ at the point $\gamma_{p_ {0}} (t)$, with $c(t)=(c_1(t),\ldots, c_p(t)) \in \mathbb{R}^p \setminus \{0\}$ and $v_0 \neq 0$. In fact, since $\Disc G = \{0\}$, we have $ \ker\dd_{\gamma_{p_{0}}(t)}(G_{|_{B^m_\epsilon \setminus V_G}}) =\textrm{T}_{\gamma_{p_{0}}(t)}G^{-1}(G(\gamma_{p_{0}}(t)))  $. If we denote $ \cT:= \lim \left( \textrm{T}_{\gamma_{p_{0}}(t)}G^{-1}(G(\gamma_{p_{0}}(t)))\right)   $, then $ \left( n_{c(t)}(\gamma_{p_{0}}(t))/\|n_{c(t)}(\gamma_{p_{0}}(t))\|\right)  \to v_0 $ implies that $ v_0 \in \cT^{\perp}  $. Thus, $ \cT^{\perp} \not \subset [T_{p_{0}} W_{p_{0}}]^{\perp}  $. Consequently, $ T_{p_{0}} W_{p_{0}} \not \subset \cT $. In other words, $\left(B^m_\epsilon \setminus V_G, W_{p_{0}}\right)$ does \textbf{not} satisfy the Thom condition $(\aa_G)$-regularity at $p_0$. Since $\mathbb{W}_1$ and $W_{p_{0}}$ were taken arbitrarily, we conclude that $p_0 \in B_{\epsilon}^{m}\cap NT_G$.

		\vspace{0.3cm}
		
		In the case of germs of maps, we are only interested in the behavior of the map in small neighborhoods of the origin. Therefore, the following theorem is an immediate consequence of the above discussion.
		
		\vspace{0.3cm}
		
		\begin{theorem}\label{tnt}
			A germ of an analytic map $ G: (\mathbb{R}^m, 0) \to (\mathbb{R}^p, 0) $, with $ m>p\ge 2 $, such that $ \Disc G = \{0 \} $, is not Thom regular at $ V_G $ if and only if there exist $ \epsilon>0 $ sufficiently small and a sequence of points $ \{p_n\} \subset B_{\epsilon}^{m}\cap NT_G $ such that $ p_n \to 0 $.
		\end{theorem}

		\begin{example} Consider the map $ G := (G_1, G_2) $ as in Example \ref{eNT1}, that is, with $ G_1(x, y, z) = x $ and $ G_2(x, y, z) = y(x^2 + y^2) + xz^2 $. Another way to prove that $ G $ is not Thom regular at $ V_G $ is by applying Theorem \ref{tnt}. Let us consider once again $ \mathbb{W} = {W_i} $ as a Whitney stratification of $ V_G $ and $ p = (0, 0, z) \in W_i $, for some stratum of positive dimension $ W_i $ with the origin in its closure. For each natural number $k$, consider $\gamma_{k,z} = \left(\frac{1}{k}, 0, z\right) \in B^{3}_{\epsilon} \setminus V_G$ and $c_{k,z} = (-kz^2, k) \in \mathbb{R}^2$.
			
			Since $\nabla G_1(x,y,z) = (1,0,0)$ and $\nabla G_2(x,y,z) = (2xy+z^2,x^2+3y^2, 2xz)$, we have $n_{c_{k,z}}(\gamma_{k,z})= (-kz^2,0,0)+(kz^2,\frac{k}{k^2},\frac{2zk}{k})$. In other words, $n_{c_{k,z}}=(0,\frac{1}{k},2z)$ is a normal direction to the fibers $G^{-1}(G(\gamma_{k,z}))$. Note that $\|n_{c_{k,z}}(\gamma_{k,z})\| = \sqrt{4z^2+\frac{1}{k^2}}$. Therefore, when $k \to \infty$, we have $$\dfrac{n_{c_{k,z}}(\gamma_{k,z})}{\|n_{c_{k,z}}(\gamma_{k,z})\|} \to (0,0,\pm1) \notin   \left( \cT_{p}W_i\right)^{\perp},$$ where again the plus or minus sign depends on the sign of $z$.\footnote{
				Now, since $\langle (0,0,\pm 1), (0, \pm 1, 0) \rangle = 0$, we have $(0,0,\pm 1) \in \cT^{\perp}$. In other words, $\cT^{\perp} \not \subset \left(\cT_{p}W_i\right)^{\perp}$. Consequently, $\textrm{T}_{p}W_i \not \subset \cT$, as previously shown in Example \ref{eNT1}.} Consequently, $p \in B_{\epsilon}^{3}\cap NT_G$. Once $p$ was taken arbitrarily, we can apply Theorem \ref{tnt} to conclude that $G$ is not Thom regular at $V_G$.

		\end{example}
		
		\vspace{0.3cm}
		
		In the paper \cite{ACT}, the authors proved a result of the \textit{Thom-Sebastiani} type to generate new examples of maps with classical Thom regularity from known examples with this property. Below, we present our adaptation of the statement.
		
		\vspace{0.3cm}
		
		\begin{proposition}\cite[Proposition 5.2]{ACT}\label{act5.2} Let $f:(\mathbb{R}^{m},0)\to (\mathbb{R}^{p},0)$ and $g:(\mathbb{R}^{n},0)\to (\mathbb{R}^{p},0)$ be germs of analytic maps in separable variables, such that both $V_f$ and $V_g$ have codimension $p$, $\Disc f = \{0\}$, and $\Disc g = \{0\}$. Suppose that $f$ and $g$ have classical Thom regularity. Then the map $G:=f+g:(\mathbb{R}^{m}\times \mathbb{R}^{n},0)\to (\mathbb{R}^{p},0)$ has classical Thom regularity.
		\end{proposition}

		\begin{proof}		
			Initially, let us show that $ \Disc G = {0} $. Indeed, if we write $x=(x_{1}, \ldots, x_{n})$ and $y=(y_{1}, \ldots, y_{n})$, we have that
			$$JG_{(x,y)}=\left[ \begin{array}{cc}
				Jf(x)&Jg(y)\\
			\end{array} \right].$$ If $(x,y) \in \Sing G$, then there exists $ i \in \{1,\ldots, p\}$ such that
			\begin{eqnarray*}
				\nabla G_{i}(x,y) = \dsum_{i \neq j =1}^{p} \lambda_{j}(x,y) \nabla G_{j}(x,y) =\dsum_{i \neq j =1}^{p} \lambda_{j}(x,y) (\nabla f_{j}(x), \nabla g_{j}(x)).
			\end{eqnarray*}
			Thus, 
			$$
			\nabla f_{j}(x) =\dsum_{i \neq j=1}^{p} \lambda_{j}(x,y) (\nabla f_{j}(x))
			$$
			and
			$$
			\nabla g_{j}(y) =\dsum_{i \neq j=1}^{p} \lambda_{j}(x,y) (\nabla g_{j}(y)).
			$$
			Hence, $x \in \Sing\,f$ e $y \in \Sing \,g$, i.e, $ (x,y) \in \Sing\, f \times \Sing\,g$.
			Therefore, $\Sing G \subseteq \Sing f \times \Sing g \subset V_{f} \times V_{g}$. On the other hand, if $(x,y) \in V_{f} \times V_{g}$, then $f(x)=g(y)=0$ and $G(x,y) = f(x)+g(y) = 0$, which means $(x,y) \in G^{-1}(0)$. Hence, $V_{f} \times V_{g} \subseteq V_G$, implying that $G$ has isolated critical values.
			
			\vspace{0.2cm}
			
			Let $\bW_1$ and $\bW_2$ be stratifications of $V_{f}$ and $V_{g}$, respectively, satisfying the Thom regularity condition. We will show that the product stratification $\bW_1 \times \bW_2$ also satisfies Thom $(a_{G})$-regularity. Indeed, consider strata $W_1 \in \bW_1$ and $W_2 \in \bW_2$ chosen such that at least one of them has a positive dimension. Let $(\alpha, \beta) \in W_1 \times W_2$ be an arbitrary point, and consider a sequence of points $(x_{i}, y_{i}) \in \R^{m} \times \R^{n}$ that converges to $(\alpha, \beta)$ and such that $G(x_{i}, y_{i}) = f(x_{i}) + g(y_{i}) \neq 0$ for all $i>0$ (i.e., $(x_{i}, y_{i}) \not\in V_G$). Now, consider the limit
			$$
			T:=\dlim_{(x_{i}, y_{i}) \to (\alpha,\beta)}T_{(x_{i}, y_{i})}G^{-1}(G(x_{i}, y_{i}))
			$$
			of tangent spaces of the fibers of $G$ along the sequence $(x_{i}, y_{i})$. We claim that $$T_{(\alpha,\beta)}(W_1 \times W_2) \subseteq T.$$
			
			\begin{itemize}
				\item \textbf{Case 1.} $f(x_{i}) \neq 0$ e $g(y_{i}) \neq 0$, $ \forall i >>1$. We claim that $T_{x_{i}}g^{-1}(g(x)) \times T_{y_{i}}g^{-1}(g(y_{i})) \subseteq T_{(x_{i}, y_{i})}G^{-1}(G(x_{i}, y_{i}))$. Indeed, let $(v_{i},w_{i}) \in T_{x_{i}}g^{-1}(g(x)) \times T_{y_{i}}g^{-1}(g(y_{i}))$. Thus there exist
				$\gamma_{i}(t) \in f^{-1}(f(x_{i}))$ and $\Gamma_{i}(t) \in g^{-1}(g(y_{i}))$ such that $\gamma(0)=x_{i}$, $\Gamma_{i}(0)=y_{i}$, $\gamma'_{i}(0)=v_{i}$ e $\Gamma'_{i}(0)=w_{i}$. Consider the curve $(\gamma_{i}(t),\Gamma_{i}(t))$. We have
				$$G(\gamma_{i}(t),\Gamma_{i}(t))=f(\gamma_{i}(t))+ g(\Gamma_{i}(t))=f(x_{i})+g(y_{i})=G(x_{i},y_{i}),$$ which implies that
				$(\gamma_{i}(t),\Gamma_{i}(t)) \in G^{-1}(G(x_{i},y_{i})) $,  so, $$ (v_{i},w_{i})=(\gamma'_{i}(0), \Gamma'_{i}(0)) \in 
				T_{(x_{i},y_{i})} G^{-1}(G(x_{i},y_{i})) .$$ Consequently, $ 
				T_{x_{i}}g^{-1}(g(x)) \times T_{y_{i}}g^{-1}(g(y_{i})) \subseteq T_{(x_{i}, y_{i})}G^{-1}(G(x_{i}, y_{i}))$.
				
				\item \textbf{Case 2.} $f(x_{i}) = 0$, $\forall  i >>1$. Since $G(x_{i}, y_{i}) \neq 0$, we have that $g(x_{i}) \neq 0$, for all $ i >> 1$. Besides,
				$x_{i} \in f^{-1}(0)=V_{f}$, for all $ i >>1$. Let $W_{x}$ be the stratum of $V_{f}$ that contains $x_{i}$. Let us show that $T_{x_{i}} W_{x_{i}} \times T_{y_{i}} g^{-1}(g(y_{i})) \subseteq T_{(x_{i}, y_{i})}G^{-1}(G(x_{i}, y_{i}))$.
				In fact, let $v \in T_{x_{i}}W_{x_{i}}$ be an arbitrary vector and let $\Delta(t)$ be a curve in $W_{x}$ with $\Delta(0)=x_{i}$ and $\Delta'(0)=v$.
				We have that  $f \circ \Delta(t)=f(\Delta(t))=0$ (because $\Delta(t) \in W_{x_{i}} \subset V_{f}$). Hence, one has
				\begin{eqnarray}\label{eqqumaestrela}
					0=\dfrac{d}{dt}_{|t=0} f \circ \Delta = d_{\Delta(0)} f(\Delta'(0))=d_{x_{i}}f(v)
				\end{eqnarray}
				Now, notice that $T_{(x_{i}, y_{i})}G^{-1}(G(x_{i}, y_{i})) = \textrm{Kern}\,d_{(x_{i}, y_{i})}G$. However, 
				$d_{(x_{i}, y_{i})}G: \R^{m} \times \R^{n} \to \R^{p}$, defined by $(v,w) \mapsto d_{(x_{i}, y_{i})} G(v,w)$ is such that
				$d_{(x_{i}, y_{i})} G(v,w) = d_{x_{i}} f(v) +d_{y_{i}}(w)$. Therefore, if 
				$v \in T_{x_{i}}W_{x_{i}}$, using \eqref{eqqumaestrela}, we have $d_{(x_{i}, y_{i})} G(v,w) = d_{x_{i}} f(v) +d_{y_{i}}(w)=0+d_{y_{i}}(w)=d_{y_{i}}(w)$.
				Since $w \in T_{y_{i}} g^{-1}(g(y_{i})) = \textrm{Ker}\,d_{y_{i}} g$, then $ d_{y_{i}} g(w)=0$. Hence, if
				$(v,w) \in T_{x_{i}} W_{x_{i}} \times T_{y_{i}} g^{-1}(g(y_{i}))$, one has $ d_{(x_{i}, y_{i})} G(v,w)=0 $, and then $ (v,w) \in \textrm{Kern}\,d_{(x_{i}, y_{i})} G= T_{(x_{i}, y_{i})}G^{-1}(G(x_{i}, y_{i}))$.
			\end{itemize}
			
			Now, since $f$ e $g$ are Thom regular, then
			$\dlim_{x_{i} \to \alpha} T_{x_{i}}f^{-1}(f(x))\supset T_{\alpha}\bW_1$ e $\dlim_{y_{i} \to \beta} T_{y_{i}}g^{-1}(g(y)) \supset T_{\beta}\bW_2$.
			Therefore, in any case, we have
			\begin{eqnarray*}
				T_{(\alpha, \beta)}(\bW_1 \times \bW_2 )&\equiv & T_{\alpha}\bW_1 \times T_{\beta}\bW_2\\
				&\subset & \dlim_{x_{i} \to \alpha} T_{x_{i}}f^{-1}(f(x)) \times 
				\dlim_{y_{i} \to \beta} T_{y_{i}}g^{-1}(g(y_{i}))\\
				&=& \dlim_{(x_{i},y_{i}) \to (\alpha,\beta)}\left(T_{x_{i}}f^{-1}(f(x)) \times T_{y_{i}}g^{-1}(g(y_{i})\right)\\
				&\subseteq &\dlim_{(x_{i},y_{i}) \to (\alpha,\beta)} G^{-1}(G(x_{i},y_{i}))=T
			\end{eqnarray*}
			
			Finally, we can refine the product stratification $\bW_1 \times \bW_2$ to a regular stratification satisfying the Whitney condition $(a)$, such that the singular set $\Sing G$ is the union of strata. By construction, this refinement is a Thom regular stratification of $V_{G}$. This completes the proof of the proposition.
		\end{proof}
		
		\vspace{0.3cm}
		
		The next result, which is a generalization of \cite[Theorem 2]{R}, shows that the converse of Proposition \ref{act5.2} is true under additional conditions.

		\vspace{0.3cm}

		\begin{theorem}\label{t2}
			Let $f:(\mathbb{R}^{m},0)\to (\mathbb{R}^{p},0)$ and $g:(\mathbb{R}^{n},0)\to (\mathbb{R}^{p},0)$ be germs of analytic maps in separable variables such that $\Disc f = \{0\}$, $\Disc g = \{0\}$, and $g$ is classically Thom regular. Suppose that the germ of the map $G:=f+g:(\mathbb{R}^{m}\times \mathbb{R}^{n},0)\to (\mathbb{R}^{p},0)$ satisfies $\Sing f \times \Sing g \subset \Sing G$. If $G$ is classically Thom regular, then $f$ is classically Thom regular.
		\end{theorem}
		
		\begin{proof}

			Assume that $f$ is not Thom regular at $V_f$. 
			We know that $G$ has an isolated critical value, i.e., $\Disc G=\{0\}$, and by hypothesis, $$V_G \cap \Sing G = \Sing G = \Sing f \times \Sing g.$$
			
			Consider any Whitney stratification $ \mathbb{W}=\{W_\alpha\} $ of $ V_G$ such that its extension to a Whitney stratification of $\mathbb{R}^m\times \mathbb{R}^n$ provides, respectively, a Whitney stratification of $\mathbb{W}_1$ and $\mathbb{W}_2$ on $\mathbb{R}^m$ and $\mathbb{R}^n$, which are stratifications for  $V_f$ and $V_g$. 
			
			Since $f$ is not Thom regular at $V_f$,  there exist  points $ p_0 \in NT_f $ with $ p_0 \to 0 $, i.e., $p_0 \in V_f \cap \Sing f$ such that for any positive dimensional strata $ W_{\alpha}^{1} $ of $\mathbb{W}_1,$ $W_{\alpha}^{1} \subset \Sing f$ containing the points $p_0$, there exist  analytic curves  $\gamma_{p_{0}} \subset \mathbb{R}^n \m V_{f}$, such that $\gamma_{p_{0}} \to p_0$ and  
			\[n_{\mu(t),f}(\gamma_{p_{0}}(t)) \to v_0 \notin [T_{p_{0}} W_{\alpha}^{1}]^{\perp} \]
			up to normalization, with $v_0 \neq 0$ and $\mu(t) \in \mathbb{R}^p \setminus \{0\}$.
			
			Now, for each $p_{0}$ one choose a point $ q_0\in \Sing g $ such that the sequence of points $q_{0} $  converge to  $0$. Thus, the pair  $z_0:=(p_0,q_0) \in \Sing f \times \Sing g =V_G \cap \Sing G $ and $ z_0 \to 0 $. Let  $W_{\beta}^{2}$ be the stratum of $\mathbb{W}_2$ which contains $q_{0}$. 	
			
			One defines the path $ \gamma_{z_{0}}:=(\gamma_{p_{0}},\gamma_{q_{0}}) $ with $ \gamma_{z_{0}} \subseteq \left( \bR^m \times \bR^n\right)  \m V_G $ and $  \gamma_{q_{0}} \to q_{0}$ which implies $\gamma_{z_{0}}\to z_0  $. One can assume without loss of generality that $ n_{\mu(t),g}(\gamma_{q_{0}}(t))$ converge for some $ v_m \in [T_{q_{0}} W_{\beta}^{2}]^{\perp}$. Hence, one has that 
			\[n_{\mu(t),G}(\gamma_{z_{0}}(t))= (n_{\mu(t),f}(\gamma_{p_{0}}(t)),n_{\mu(t),g}(\gamma_{q_{0}}(t))) \to (v_0 , v_m) \]

			Consider $W_{\alpha,\beta}:=W_{\alpha}^{1} \times W_{\beta}^{2} $. Hence, 
			\[T_{z_{0}}W_{\alpha, \beta}=T_{z_{0}}\left( W_{\alpha}^{1} \times W_{\beta}^{2}  \right) =T_{p_{0}} W_{\alpha}^{1} \times T_{q_{0}} W_{\beta}^{2}.\]
			Given any $ (w,0) \in T_{z_{0}}W_{\alpha,\beta}$ one has  $ \left\langle (v_0, v_m), (w,0)  \right\rangle = \left\langle v_0, w \right\rangle$. Since $v_0 \notin [T_{p_{0}}W_{\alpha}^{1}]^{\perp}$, there exists some $w_0 \in T_{p_{0}} W_{\alpha}^{1}$ such that $ \left\langle v_0, w_0 \right\rangle\neq 0$.  	Therefore,  $ (w_0,0) \in T_{z_{0}} W_{\alpha,\beta} $ satisfies $  \left\langle (v_0, v_m), (w_0,0) \right\rangle \neq 0 $ which implies 
			$ (v_0, v_m) \notin [T_{z_{0}} W_{\alpha,\beta}]^{\perp}$. Consequently, the sequence of points  $ z_0 $ belongs to $ NT_G $ and since it converges to $ 0 $,  $ G $ is not Thom regular at $ V_G $.
		\end{proof}
		
		\vspace{0.3cm}
		
		\begin{remark}
			Let $f:(\mathbb{R}^{m},0)\to (\mathbb{R}^{p},0)$ and $g:(\mathbb{R}^{n},0)\to (\mathbb{R}^{p},0)$ be germs of analytic maps in separable variables. Suppose that $f$ satisfies $\Sing f \subset \left\lbrace x\in \mathbb{R}^m \,|\, Jf(x) = 0 \right\rbrace$. In this case, we have $\Sing f \times \Sing g \subset \Sing G$. Indeed, given $(x,y) \in \Sing f \times \Sing g$, there exist $\lambda_j \in \mathbb{R}$ with $j=1,\ldots,p$, such that $\sum \lambda_j^2 \neq 0$ and $\sum \lambda_j \nabla g_j(x) = 0$. Therefore, $\sum \lambda_j \nabla G_j(x) = \left( \sum \lambda_j \nabla f_j(x), \sum \lambda_j \nabla g_j(y)\right) = \left(0,0\right)$. Hence, $(x,y) \in \Sing G$. The same statement holds if we replace $f$ with $g$.
		\end{remark}

		\begin{definition}
			Let $G:(\mathbb{R}^{m},0)\to (\mathbb{R}^{p},0)$ be a germ of an analytic map, and let $M$ be a connected component of positive dimension of $ \Sing G \cap V_G $ such that $0 \in \overline{M}$. We say that $M$ is \textit{$W_G$-invariant} if, for any Whitney stratification\footnote{Actually, as "any Whitney stratification" we mean "any reduced Whitney stratification" in the following sense. Reduced means: \emph{Let $ \mathbb{W}=\{W_\alpha\} $ be a Whitney stratification and suppose $ W_1 $ and $ W_2 $ are two strata such that $ W_2 \subset \overline{W_1} $ and $ W_1 \cup W_2 $ is also a smooth manifold so that two strata can be put one keeping regularity of Whitney stratification.}} of $ V_G $, $ M \setminus \{0\} $ is a unique stratum, i.e., $ M \setminus \{0\} $ does not contain any sub-stratum.
		\end{definition}
		
		The next proposition provides a new criterion for determining whether a map does not have classical Thom regularity.
		
		\begin{proposition}\label{p1}
			Let $G:(\mathbb{R}^{m},0)\to (\mathbb{R}^{p},0)$ be  an analytic map germ, and let $M$ be a connected component of positive dimension of $ \Sing G \cap V_G $ that is $ W_G $-invariant. If there exists a sequence of points $x_\nu$ in $M$ such that $x_\nu \to 0$ and for every sufficiently large $\nu$, there exists an analytic curve $\gamma_{x_\nu} \subset B^{2n}_{\e} \m V_{G}$, such that  $\gamma_{x_\nu} \to x_\nu$ and  
			\[n_{c(t)}(\gamma_{x_\nu}(t)) \to v_\nu \notin [T_{x_\nu} M]^{\perp}, \]
			with $ v_\nu \neq 0 $, $ c(t) \in \mathbb{R}^m \m \{0\} $ and $ n_{c(t)} $ is a normal vector to the fibers of $G$ at the points $\gamma_{x_\nu}(t)$, then $G$ does not have classical Thom regularity.
		\end{proposition}
		\begin{proof}
			Since $ M $ is $ W_G-$invariant one has that $ x_\nu \in NT_G $ for all big enough $ \nu $.
		\end{proof}

		\begin{lemma}\label{l1}
			Let $G:(\mathbb{R}^{m},0)\to (\mathbb{R}^{p},0)$ be an analytic map germ, and let $M$ be a $ 2 $-dimensional connected component    of $ \Sing G \cap V_G $ such that $0 \in \overline{M}$. Let $\gamma: \left[ 0,1\right] \to M$ be an analytic curve  with $\gamma \to 0$.  
			If for each $t \in (0,1)$, there exists a curve $\beta_t: [0,1] \to M$, with $\beta_t(0) = \beta_t(1) = \gamma(t)$, containing the origin in its interior, 
			then $M$ is $W_G$-invariant.
		\end{lemma}
		\begin{proof}
			Since $\gamma \to 0$, $M$ contains a $2$-disk (up to homeomorphism),  $M \setminus \{0\}$ is a unique stratum.
		\end{proof}

		
		\begin{example}
			Let $f:(\mathbb{C}^n,0) \to (\mathbb{C},0)$ be a germ of a polar weighted mixed function. If $M \subset \Sing f \cap V_f$ is a complex line, then it is $W_f$-invariant. In fact, $M$ is a real plane, and the orbit of any fixed point $p_0 = (0,\ldots,z_j,\ldots,0) \in M \setminus \{0\}$ under the action of $S^1$ is contained in $M$. 
			The statement follows from Lemma \ref{l1}.
		\end{example}
		
		\vspace{0.2cm}


		\begin{remark}\label{r1}
			The technique used in Proposition \ref{p1} was originally presented in \cite{PT}  for the family $f_k(x,y,z)=(x+z^k)\bar{x}y$, with  $k>1$, and adapted to the more general mixed functions by Ribeiro in \cite{R}. In our notation, if $ M=\left\lbrace x=z=0\right\rbrace $, then it is $W_{f_{k}}-$invariant for all $ k>1 $, see \cite[section \S 5.1]{PT}. 
			
		\end{remark}

		\section{The Milnor condition (b)}

		\subsection{The Milnor Set}
		
		We begin by considering again $U \subset \bR^m$ an open set containing the origin $0$, and $\rho: U \to \bR_{\ge 0}$ a non-negative proper function defining the origin. 
		
		\vspace{0.2cm}
		
		\begin{definition} \index{Milnor set} \label{d:M}
			Let $G:(\bR^m, 0) \to (\bR^p,0)$ be a germ of analytic map. We denote by \[M_\rho(G):=\left\lbrace x \in U \mid \rho \not\pitchfork_x G \right\rbrace  \]
			
			\noindent the set of points in a neighborhood of the origin where $ G $ and $ \rho $ do not intersect transversely. The set $ M_\rho(G) $ is called the \textit{$\rho$-non-regular points} of $G$ or the \textit{Milnor set} of $G$. \index{Milnor set}
		\end{definition}
		
		\vspace{0.2cm}
		
		The following proposition provides a mechanism to compute the Milnor set of a map $G$.
		
		\vspace{0.2cm}
		
		\begin{proposition}\cite{Ha}\label{pm}
			Let $G:(\bR^m, 0) \to (\bR^p,0)$ be a germ of analytic map, and let $U \subset \bR^m$ be an open set containing the origin $0$. A point $x\in U$ belongs to the Milnor set $M_\rho(G)$ if and only if the vectors $\{\nabla \rho(x), \nabla G_1 (x),\ldots, \nabla G_p (x) \}$ are linearly dependent over $\bR$.
		\end{proposition}
		
		\begin{proof}
			Note that if $x_0\in U$ such that $x_0=0$, then the map $\rho$ is not a submersion at $x_0$. Therefore, $x_0$ belongs to $M_\rho(G)$ if and only if $\{\nabla \rho(x_0), \nabla G_1 (x_0),\ldots, \nabla G_p (x_0) \}$ are linearly dependent over $\R$.
			
			We now assume that $x_0 \neq 0$. If $x_0 \notin M_\rho(G)$, then $G$ is a submersion by definition. Consequently, ${ \nabla G_1 (x_0),\ldots, \nabla G_p (x_0) }$ are linearly independent. Suppose that $x_0, \nabla G_1 (x_0),\ldots, \nabla G_p (x_0)$ are linearly dependent. Since the gradients generate the normal space to the fibers $G^{-1}(G(x_0))$, and $x_0$ is spanned by these gradients, $x_0$ is also normal to the fiber $G^{-1}(G(x_0))$. In other words, $x_0$ is normal to the tangent space $\textrm{T}_{x_0}, G^{-1}(G(x_0))$. On the other hand, we have assumed that $x_0 \notin M_\rho(G)$, so $ \textrm{T}_{x_0}\, G^{-1}((G(x_0)))\oplus \textrm{T}\,S^{m-1}_{\|x_0\|}=\R^m $. Thus, $x_0$ is normal to every vector in $\R^m$, which means $x_0 = 0$, leading to a contradiction. Therefore, $\{\nabla \rho(x_0), \nabla G_1 (x_0),\ldots, \nabla G_p (x_0) \}$ is linearly independent over $\R$.
			
			Conversely, assume that ${\nabla \rho(x_0), \nabla G_1 (x_0),\ldots, \nabla G_p (x_0) }$ are linearly independent over $\R$. In this case, $ G $ and $ \rho $ are both submersions at $ x_0 $. Since $\textrm{dim}\,\left( \textrm{T}\,S^{m-1}_{\|x_0\|}\right) = m-1 $, if $\textrm{T}_{x_0}\, G^{-1}((G(x_0)))$ and  $\textrm{T}\,S^{m-1}_{\|x_0\|}$ are not transverse, then $\textrm{T}_{x_0}\, G^{-1}((G(x_0))) \subset \textrm{T}\,S^{m-1}_{\|x_0\|}$ and consequently, $ \textrm{span}\{x_0\}= \left( \textrm{T}\,S^{m-1}_{\|x_0\|}\right) ^{\perp} \subset \left( \textrm{T}_{x_0}\, G^{-1}((G(x_0)))\right) ^{\perp} $, i.e, $$ x_0\in \textrm{span}\{\nabla G_1 (x_0),\ldots, \nabla G_p (x_0) \}  ,$$\noindent  leading to a contradiction. Therefore, $ G $ and $ \rho $ are transverse at $ x_0 $, i.e., $ x_0 \notin M(G) $.		
		\end{proof}
		
		\vspace{0.3cm}
		
		Consider an analytic germ $G:(\bR^m, 0) \to (\bR^p,0)$ and $U \subset \bR^m$ an open set containing the origin $0$. For each $x\in U$, define the matrix
		
		\begin{center}
			$A(x):= \left[ \begin{array}{c}
				\nabla G_{1}(x) \\ 
				\vdots \\ 
				\nabla G_{p}(x) \\ 
				\nabla \rho (x)
			\end{array}\right] $
		\end{center}
		
		\vspace{0.2cm}
		
		If the matrix $A(x)$ is square, then from Proposition \ref{pm}, it follows that $x\in M_\rho(G)$ if and only if $ \textrm{det}\left( A(x)\right) =0 $. More generally, the Lagrange identity guarantees that $$ M_\rho(G)=\{x\in U \,|\, \det\left( A(x)A(x)^T\right) =0\}, $$ where $ A(x)^T $ denotes the transpose of the matrix $ A(x) $.

		\vspace{0.2cm}

		\begin{example}\label{mfx1}
			Let $G:(\mathbb{R}^3,0) \to (\mathbb{R}^2,0)$ be given by $G(x,y,z)=(xy,xz)$, where $(x,y,z) \in \R^3$. To calculate the Milnor set $M(G)$, we consider the matrix $$A(x,y,z):=\left[ \begin{array}{ccc}
				y&x&0\\
				z&0&x\\
				x&y&z\\
			\end{array} \right]. $$
			
			The Milnor set $M(G)=\{(x,y,z)\in\mathbb{R}^3 \mid \det\left(A(x,y,z)\right) = 0 \}$.
			Consequently, $M(G) = \{x=0\} \cup\{x^2-y^2-z^2=0\}$.
		\end{example}
		
		\vspace{0.2cm}
		
		\vspace{0.2cm}

		\begin{example}\label{ex1}

			Consider again the map $ G:=(G_1,G_2,G_3,G_4):(\mathbb{R}^6,0) \to (\mathbb{R}^4,0) $ given by  
			
			$$ \left\lbrace \begin{array}{ccl}
				G_1(x,y,z,w,a,b)&=x^2z+y^2z &   \medskip\\
				G_2(x,y,z,w,a,b)&= wx^2+wy^2&  \medskip\\
				G_3(x,y,z,w,a,b)&= ax^2+ay^2&  \medskip\\
				G_4(x,y,z,w,a,b)&= bx^2+by^2&  \medskip\\
				
			\end{array}\right. $$
			
			We have 
			$$\left \lbrace \begin{array}{lcl}
				V_G & = & \{x^2+y^2=0\} \cup \{a^2+b^2+w^2+z^2=0\} \medskip\\
				\Sing G & =& \{x^2+y^2=0\}  \medskip\\
			\end{array}\right.$$
			
			Now, let us calculate the Milnor set of $ G $. Consider the matrix:
			
			{\[\displaystyle A(\textbf{v})= \left[ \begin {array}{cccccc} 2\,xz&2\,yz&{x}^{2}+{y}^{2}&0&0&0\\ \noalign{\medskip}2\,wx&2\,wy&0&{x}^{2}+{y}^{2}&0&0\\ \noalign{\medskip}2\,xa&2\,ya&0&0&{x}^{2}+{y}^{2}&0\\ \noalign{\medskip}2\,xb&2\,yb&0&0&0&{x}^{2}+{y}^{2}\\ \noalign{\medskip}x&y&z&w&a&b\end {array} \right] \]}
			\noindent
			where $ \textbf{v}:=(x,y,z,w,a,b) $. Now, after finding the square matrix,  $ A(\textbf{v})A(\textbf{v})^T $, we can calculate its determinant:
			{\[\displaystyle \det\left( A(\textbf{v})A(\textbf{v})^T\right) = \left( {x}^{2}+{y}^{2} \right) ^{7} \left( 2\,{a}^{2}+2\,{b}^{2}+2\,{w}^{2}-{x}^{2}-{y}^{2}+2\,{z}^{2} \right) ^{2}.\]} Therefore,   $ M(G) =  \{x^2+y^2=0\} \cup \{2a^2+2b^2+2w^2-x^2-y^2+2z^2=0\} $.
		\end{example}
		
		\vspace{0.2cm}
		
		Note that in the previous example, we have $\Sing G \subset M\left(G \right)$. The next result, which is an immediate consequence of Proposition \ref{pm}, shows that this condition is always true.
		
		\vspace{0.2cm}
		
		\begin{corollary}
			If $G:(\bR^m, 0) \to (\bR^p,0)$ is an analytic map germ then $M_\rho(G) =\Sing (G,\rho)$, where  $(G,\rho): U\to \bR^{p+1}$. In particular, $\Sing G \subset M_{\rho}\left(G \right).$
		\end{corollary}

		\vspace{0.2cm}

		\begin{remark}
			The Milnor set for a mixed function is given by $$M(G)= \left\lbrace z\in \bC^n\,|\,\exists\,\lambda\in\bR,\mu\in\bC^{\ast}, \|\mu\|=1 \textrm{ s.t. } \lambda z = \mu\overline{\textrm{d}G}(z,\bar{z})+\bar{\mu}\bar{\textrm{d}}G(z,\bar{z} )\right\rbrace.$$
		\end{remark}

		\vspace{0.2cm}
		
		\subsection{Milnor Condition (b) and $ \rho $-regularity}
		
		\vspace{0.2cm}
		
		\begin{definition}
			We say that an analytic germ $G: (\bR^m,0) \to (\bR^p,0)$ satisfies the \textit{Milnor condition (b) at the origin} if:
			\begin{equation}\label{eq:mainclass}
				\overline{M(G) \setminus V_G} \cap V_G \subseteq \{0\},
			\end{equation}
			where the closure of the set $M(G) \setminus V_G$ is understood as a germ of a set at the origin.
		\end{definition}

		\vspace{0.2cm}

		
		As explained in \cite{R}, condition \eqref{eq:mainclass} is equivalent to the existence of a conical neighborhood of $ V_G \setminus \{0\} $ where there are no points from the Milnor set, except possibly those included in $ V_G $. Thus, this condition implies $ \rho $-regularity for $ G $. Conversely, the $ \rho $-regularity of $ G $ gives rise to a conical neighborhood of $ V_G \setminus \{0\} $ without non-regular points, and therefore, condition \eqref{eq:mainclass} holds.
		
		\vspace{0.2cm}

		\begin{example}\label{exaa} Consider again the map $  G:(\mathbb{R}^3,0) \to (\mathbb{R}^2,0) $ given by $ G(x,y,z)=(xy,xz)$  on the Example \ref{mfx1}. We know that 
			$$\left \lbrace \begin{array}{lcl}
				V_G & = & \{x=0\} \cup \{y=z=0\} \medskip\\
				M(G) & = & \{x=0\} \cup \{x^2-y^2-z^2=0\} \medskip
			\end{array}\right.$$ Let us show that $ G $  satisfies the condition  \eqref{eq:mainclass}. Indeed, let $p_0=(x_0,y_0,z_0)\in \overline{M(G)\m V_G}\cap V_G$. There exists a sequence  $p_n:=(x_n,y_n,z_n)\in M(G)\m V_G$ such that $p_n \to p_0$ and $x_{n}^{2}=y_{n}^{2}+z_{n}^{2}$ with $y_{n}^{2}+z_{n}^{2} \neq 0$. Since  $p_0\in V_G$, we need to consider two cases:
			
			\noindent
			\textbf{Case 1:} $p_0=(0,y_0,z_0)$. Then $$0=\lim x_{n}^{2}= \lim (y_{n}^{2}+z_{n}^{2})=y_{0}^{2}+z_{0}^{2}.$$
			
			\noindent
			\textbf{Case 2:} $p_0=(x_0,0,0)$. Then $$x_{0}^2=\lim x_{n}^{2}= \lim (y_{n}^{2}+z_{n}^{2})=y_{0}^{2}+z_{0}^{2} = 0.$$
			In any case, we have $p_0=0$, and $G$ satisfies condition \eqref{eq:mainclass}, as we wanted to show.
			
		\end{example}

		\vspace{0.2cm}
		
		\begin{example}\label{mhx1}
			Consider the map $G:(\mathbb{R}^3,0)\to (\mathbb{R}^2,0)$ given by $G(x,y,z)=(xy,z^2)$. We have

			\noindent $$\left \lbrace \begin{array}{ll}
				V_G =\{x=z=0\}\cup \{y=z=0\}\medskip\\
				M(G)= \{x=\pm y\}\cup\{z=0\} \medskip\\
			\end{array}\right.$$
			
			\vspace{0.2cm}
			
			
			\noindent
			We claim that $V_G \subset \overline{M(G)\m V_G}\cap V_G$. Indeed, write $V_G=V_1 \cup V_2$ where $V_1=\{x=z=0\}$ e $V_2=\{y=z=0\}$. Let $q_1=(0,y_1,0) \in V_1$ such that  $y_1 \neq 0$. Consider a sequence of points  $q_n:=(1/n,y_1,0)$. We have that $q_n \to q_1$ and $q_n \in M(G)\m V_G$. Thus, $q_1 \in \overline{M(G)\m V_G}\cap V_G$, i.e., $V_1\m\{0\} \subset  \overline{M(G)\m V_G}\cap V_G$. Analogously, one can shows that  $V_2\m\{0\} \subset \overline{M(G)\m V_G}\cap V_G$. Now, if $q_0=0$ one can consider the sequence $q_n:=(1/n,1/n,0)$. Hence, $V_G \subset  \overline{M(G)\m V_G}\cap V_G $ which implies that  $\overline{M(G)\m V_G}\cap V_G =V_G \supsetneq\{0\}$. Therefore, $G$ does not satisfy the condition \eqref{eq:mainclass}.
		\end{example}
		
		\vspace{0.2cm}

		\vspace{0.2cm}

		The following lemma shows that the Milnor condition (b) at the origin implies the $ \rho $-regularity of the map $ G $.
		
		\begin{lemma}\label{f2L3}
			Let $G:(\mathbb{R}^m,0)\to (\mathbb{R}^p,0)$ be an analytic germ with an isolated critical value. The map $G$ satisfies the Milnor condition (b) if and only if there exists $\e_0>0$ such that for any $0<\e <\e_0$, there exists $\eta$, $0<\eta \ll \e$, such that the restriction map \eqref{rest1}:
			\begin{equation*}
				G_{|}: S_{\e}^{m-1}\cap G^{-1}(B^{p}_{\eta} \setminus \{0\}) \to B^{p}_{\eta} \setminus \{0\}
			\end{equation*}
			is a smooth submersion.
		\end{lemma}

		\begin{proof}
			Let $\e_{0}>0$ be such that for any $0< \e <\e_0$, condition \eqref{eq:mainclass} is satisfied. Then we can find a conical neighborhood $\mathcal{N}$ of $V_G \setminus {0}$ such that $(M(G)\setminus G^{-1}(\Disc G)) \cap \mathcal{N} = \emptyset$. Consequently, for all $x\in \left(\mathcal{N} \cap S_{\e}^{m-1}\right) \setminus G^{-1}(\Disc G)$, we have $S_{\e}^{m-1} \pitchfork_x G^{-1}(G(x))$. Now, consider $\eta_\e >0$ depending on $\e>0$ such that the inclusion $S_{\e}^{m-1}\cap G^{-1}(B^{p}_{\eta_\e} \m \Disc  G) \subset \left(\mathcal{N} \cap S_{\e}^{m-1}\right) \m G^{-1}(\Disc G)$. Then, for any $0<\eta <\eta_\e$, we have that (\ref{rest1}) is a smooth submersion.
			
			Conversely, let $\e_0>0$ be such that for any $0<\e <\e_0$, there exists $\eta$, $0<\eta \ll \e$, such that the restriction map (\ref{rest1}) is a smooth submersion. In this case, we have $M(G) \cap ( S_{\e}^{m-1}\cap G^{-1}(B^{p}_{\eta} \setminus \Disc G)) = \emptyset$. Taking $\e \to 0$, we obtain a conical neighborhood $\mathcal{N}$ of $V_G \setminus {0}$ such that $(M(G)\setminus G^{-1}(\Disc G)) \cap \mathcal{N} = \emptyset$, which is equivalent to condition \eqref{eq:mainclass}.
		\end{proof}
		
		\vspace{0.2cm}
		
		Now, let us show that the condition $ \Sing G \cap V_G = \{0\} $ implies condition \eqref{eq:mainclass}. Before that, let us consider the following lemma:
		
		\vspace{0.2cm}
		
		\begin{lemma}\label{topl1}
			Let $ G:(\mathbb{R}^n, 0) \to (\mathbb{R}^p,0) $ be an analytic germ. If $ \Sing G \cap V_G \subset {0} $, then there exists $ \e_0>0 $ such that for all $ 0<\e \le \e_0 $, $ 0\in \mathbb{R}^p $ is a regular value of the restriction 	\[G_{|}:S^{n-1}_{\e} \to \mathbb{R}^{p} \]
		\end{lemma}
		
		\begin{proof}
			Since $ \Sing G \cap V_G \subset {0} $, there exists $ \e_1>0 $ such that for every $ 0<\e<\e_1 $, the set $ M_\e:=V_G \m {0}\cap B_{\e}^n $ is a smooth manifold of dimension $ m-p $. Let $ \rho_{|}:M_\e \to \mathbb{R} $ be the function defined by $ \rho_{|}(x)=|x|^2 $. According to \cite[Lemma 2.7]{Mi}, we have $ \Sing (\rho_{|}) = M_\e \cap W $, where $ W $ is the set of points $ x \in V_G \cap B^{n}_{\e} $ such that the matrix
			
			\begin{equation*}
				\left[ 	\begin{tabular}{c}
					$\nabla G_1 (x) $ \\
					$ \vdots $ \\
					$\nabla G_p (x) $\medskip\\
					$\nabla  \rho (x)$ \medskip
				\end{tabular}\right] 
			\end{equation*}	has rank $ \le p $. We claim that $ 0 \notin \overline{W} $. Indeed, if $ 0 \in \overline{W} $, then by the Curve Selection Lemma, there exists a real analytic curve $ \alpha:(0,\delta ) \to \mathbb{R}^n $ such that $ \alpha(0)=0 $ and $ \alpha(t) \in W \m {0} $ for all $ t \in (0,\delta ) $. Consider the composition $ \rho \circ \alpha:[0, \delta ) \to \mathbb{R}$. For all $ t \in (0,\delta )$, we have $ \nabla\rho (\alpha(t)) = 0 $ (since $ \alpha(t) \in \Sing (\rho_{|}) $), thus  $$ \dfrac{ (\rho \circ \alpha)(t)}{ t} = 0, $$ implying $ \rho \circ \alpha \equiv 0 $. Consequently, $ \alpha(t) = 0 $ for all $ t \in (0,\delta )$, which is a contradiction. Therefore, there exists $ \e_0 >0 $ with $ \e_0 < \e_1 $ such that $ W \cap B_{\e_0}^n = {0} $.
			
			Now, fix an arbitrary $ \e>0 $ such that $ \e <\e_0 $, and suppose that $ 0 $ is a critical value of $ G_{|} $. By definition, there exists a singular point $ x \in (G_{|})^{-1}(0) = S^{n-1}_{\e} \cap V_G $. Since  $ \Sing G \cap V_G  \subset \{0\}  $, then  $ x \in \mbox{span} \{\nabla G_1, \ldots, \nabla G_p \} $, i.e., $ x\in W\cap B^{n}_{\e_0} $, which is a contradiction again.
		\end{proof}
		
		\vspace{0.2cm}
		
		\begin{proposition}\label{topp1}
			Let $G:(\mathbb{R}^m, 0) \to (\mathbb{R}^p,0)$ be a real analytic germ. If $ \Sing G \cap V_G  \subset \{0\} $, then $ G $ satisfies the Milnor condition (b), i.e., $ \overline{M(G) \m V_G} \cap V_G \subset \{0\} $.
		\end{proposition}
		
		\begin{proof}
			
			We claim that for  $ \e_0>0 $ given in the Lemma \ref{topl1} and  $0<\e\le \e_0$, there exists $ \eta(\e)>0 $ with $ 0<\eta(\e) \ll \e $, such that for  $ 0<\eta<\eta(\e) $ the $ p $-dimensional disc $ B^{p}_\eta $ does not contain critical values of the restriction  
			\[G_{|}:S^{m-1}_{\e} \to \mathbb{R}^{p}. \]
			Indeed, let $ \e>0 $ with $0<\e\le \e_0$ fixed. If the claim is false, it follows from Lemma  \ref{topl1},    that for  $ n\in \mathbb{N} $ sufficiently large, there exists $ y_n \in \Disc (G_{|}) \cap (B^{p}_{1/n} \m \{0\} )$, where $ \Disc (G_{|}) = G_{|}(\Sing (G_{|})) $. Let $ x_n \in \Sing (G_{|}) $ be such that  $ y_n = G_{|}(x_n) $. Since  $ y_n \to 0 $ as $ n \to \infty $, and $ G_{|} $ is proper, then $ x_n $  converges for some $ x_0 $  up to a subsequence. By continuity
			\[0=\lim_{n \to \infty} G_{|}(x_n)= G_{|}(x_0) . \]
			Thus, $ x_0 \in \overline{\Sing (G_{|})}=\Sing (G_{|}) $ e $ x_0 \in (G_{|})^{-1}(0) $ i.e., $ 0 \in \mathbb{R}^p $ is a critical value of $ G_{|}, $ which is a contradiction. 
			
			Consequently, the new restrictions
			\begin{equation}\label{006}
				G_{|}:S^{m-1}_{\e} \cap G^{-1}(B_{\eta}^{p})\to B_{\eta}^{p}
			\end{equation}
			and
			\begin{equation}\label{007}
				G_{|}:S^{m-1}_{\e} \cap G^{-1}(B_{\eta}^{p}\m \{0\})\to B_{\eta}^{p} \m \{0\}
			\end{equation}
			are submersions. Now the result is a consequence of Lemma \ref{f2L3}.
		\end{proof}
		
		\vspace{0.2cm}

		\begin{remark}\label{mhl1}
			Let $f:(\bR^m,0) \to (\bR^{p},0)$ be a germ of real analytic map and $g:(\bR^n,0) \to (\bR^n,0)$ be a germ of diffeomorphism such that $f$ and $g$ have separable variables. Consider the germ $ G $ defined from $ f $ and $ g $ as follows: $G:=(f,g):(\bR^{m}\times \bR^n,0) \to (\bR^{p}\times \bR^n,0)$. If $f$ has an isolated singularity at the origin, then $\Sing G \cap V_G =\{0\}$. Indeed, we have $\Sing G = \Sing f \times \bR^n=\{0\} \times \bR^n$ and $V_G = V_f \times V_g= V_f \times \{0\}$. Thus, $\Sing G \cap V_G = (\{0\} \times \bR^n) \cap V_f \times \{0\} = \{0\}$. By Proposition \ref{topp1}, $ G $ satisfies the condition \eqref{eq:mainclass}.
		\end{remark}
		
		\vspace{0.2cm}

		\begin{example}
			Let $f: \mathbb{R}^n \to \mathbb{R}$ be given by $f(x_1,\ldots,x_n)=\sum_{j=1}^{n}c_jx_{j}^{a_{j}}$ with $a_j \ge 2$, $c_j \in \bR \setminus \{0\}$, and let $g:(\bR^n,0) \to (\bR^n,0)$ be the identity map $g(y)=y$. Since $\Sing f = \{0\}$, it follows that $ G $ satisfies the condition \eqref{eq:mainclass}.
		\end{example}
		
		\vspace{0.2cm}
		
		\subsection{The Milnor condition (b) for composed singularities}
		
		Let $F:\bR^{m}\rightarrow \bR^{p}$ and $ G:\bR^{p}\rightarrow \bR^{k}, m\geq p\geq k\geq 2$ be two analytic maps germ. Denoted by $H=G\circ F$ the map composition between $F$ and $G$. Unless otherwise stated, we are going to keep this notation fixed throughout this section. 
		
		\vspace{0.2cm}
		
		It follows from Theorem \ref{ttf} and Lemma \ref{f2L3} that an analytic map germ admits a Milnor tube fibration if the discriminant set is just the origin and it satisfies the Milnor conditions (b). In \cite{ADRS}, the authors consider the following natural question:  
		
		\begin{question}
			What conditions must have $F$ and $G$ to guarantee that, the composition map $H$ satisfies the Milnor conditions (b) and $ \Disc H = \{0\} $?
		\end{question}
		
		About the Milnor condition $ \Disc H = \{0\} $, the problem was treated entirely in \cite{ADRS} by the authors. Here, we are going to give special attention to the problem of the Milnor condition (b) for map composition. 
		
		\vspace{0.2cm}
		
		In \cite{CT23}, the authors proved that $H$ satisfies the Milnor condition (b), provided that $F$ satisfies the Milnor condition (b) and $G$ has an isolated singular point, i.e, $\Sing G\subseteq \{0\}$. Before this moment, we can find other approaches to this subject in the literature, just in the particular case when $G$ is the canonical projection $\pi:\bR^{p}\rightarrow \bR^{k}$. For $F$ with an isolated singular point, the case of which was initially approached by Y. Iomdin \cite{Io}, to resolve a conjecture established by J. Milnor. \cite[p.100]{Mi}. Furthermore, this case ($\Sing F \subseteq\{0\}$) was approached in \cite{DA} by using completely independent tools and techniques, providing formulas for the Euler characteristic of the Milnor Fiber. Subsequently in \cite{DA}, the authors addressed the more general case, where $F$ satisfies the Milnor condition (b), $ \textrm{dim} \Sing F\ge 0 $ with $ \Disc F=\{0\} $, and proved that the map $\pi \circ F$ satisfies them too, see \cite[Lemma 4.1]{DA}.
		
		\vspace{0.2cm}
		
		In general, it is not true the map composition $H=G\circ F$ satisfies the Milnor condition (b) when $\dim \Sing G>0$. Even if $F$ and $G$ satisfy the Minor condition (b), it is not true their composition holds it, as was shown lately in \cite[Example 4.2]{ADRS}. Let us introduce another example where the Milnor condition (b) fails for the map $H$ under the same conditions for components $ F $ and $ G $.

		\vspace{0.2cm}
		
		\begin{example}\label{contraexamplo}
			Let $F:\bR^{4}\rightarrow \bR^{3}$ and $ G:\bR^{3}\rightarrow \bR^{2}$ be real analytic maps germ giving by $F(x,y,z,w):=(x,y,z(x^{2}+y^{4}+z^{6}))$ and $G(uv,vt)$. Therefore, $H(x,y,z,w)=(x y,y z(x^{2}+y^{4}+z^{6}))$, $\Sing H=\{y=0\}$ and $M(H)=\{y=0\}\cup\{w=0,p(x,y,z)=0\}$, where $p(x,y,z)=x^{4}-x^{2}y^{2}+x^{2}y^{4} - y^{6}-x^{2}z^{2}+5y^{4}z^{2}+7x^{2}z^{6}-
			7y^{2}z^{6}+z^{8}$. In this case, one has that $F, G$, and $H$ have isolated critical value. Moreover, $F$ and $G$ satisfy Milnor condition (b), but $H$ does not. To see that $H$ does not hold Milnor condition (b), we can procedure in the following way: Setting $b(y,z):=-y^{2}+y^{4}-z^{2}+7z^{6}$ and $c(y,z):=-y^{6}+y^{2}(5y^{2}z^{2}-7z^{6})+z^{8}$, then 
			
			\begin{center}
				$p(x,y,z)=x^{4}+b(y,z)x^{2}+c(y,z)=0$
			\end{center}
			
			By continuity, it is possible takes $\epsilon>0$ such that for all $y,z\in \bR$ with $\epsilon>|y|,|z|>0$ one has that $b(y,z)^{2}-4c(y,z)\geq 0$. Consequently, the branch

			\begin{center}
				$x(y,z)=\sqrt{\frac{-b(y,z)+\sqrt{b(y,z)^{2}-4c(y,z)}}{2}}$
			\end{center}
			
			is well defined for all $y,z\in \bR$ with $\epsilon>|y|,|z|>0$.
			
			By construction, the branch $\phi(y,z):=(x(y,z),y,z,0)\subset M(H)\setminus \Sing H$ for all $y,z\in \bR$ with $\epsilon>|y|,|z|>0$. Consequently $\phi(0,z)\in \overline{ M(H)\setminus \Sing H}\cap\Sing H$, for all $z\in \bR$ with $|z|<\epsilon$ and therefore $H$ does not hold Milnor condition (b).
			
		\end{example}
		
		\vspace{0.2cm}
		
		The example \ref{contraexamplo} is inspired in example 4.2 in \cite{ADRS} and we can project $M(H)$ in $\bR^{3}$ to illustrate why $H$ does not satisfy Milnor condition (b). See figure \ref{figura}.

		\vspace{0.2cm}

		\begin{figure}\label{figura}
			\centering
			\includegraphics[scale=0.45]{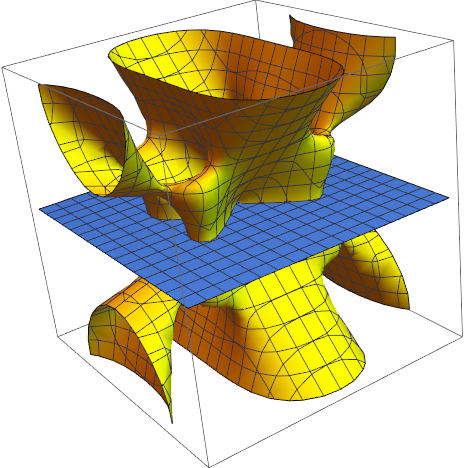}
			\caption{The Milnor set $M(H)$ in example \ref{contraexamplo}. In Yellow, the set $\{w=0,p(x,y,z)=0\}$ and, in blue the set $\Sing H$.}
			\label{fig:enter-label}
		\end{figure}
		
		\vspace{0.2cm}
		
		In the same work (\cite{ADRS}), the authors introduced a new regularity condition and proved the following. 
		
		\vspace{0.2cm}
		
		\begin{theorem}\cite[Theorem 4.4]{ADRS}\label{tp}
			If  $F$ and $ G $ have isolated critical value and $ F $ satisfies the Milnor condition (b),  the  composition map $H=G\circ F$ satisfies the Milnor condition (b) if and only if the following condition holds 
			
			\vspace{0.2cm}
			
			\begin{equation}\label{regularity condition}
				\overline{F(M(H)\setminus \Sing H)}\cap \Sing G\subseteq \{0\}. 
			\end{equation}
		\end{theorem}
		
		\vspace{0.2cm}
		
		This new condition is sharp \cite[Example 4.9]{ADRS}. Also, the authors showed that $H$ can satisfy the Milnor condition (b) even if $F$ or $G$ does not satisfy it (\cite[Example 4.10 and Example 4.11]{ADRS}). Furthermore, they provided an example where $F, G$, and $H$ hold the Milnor condition  (b), have isolated critical value, and $\dim \Sing G>0$. 
		
		\vspace{0.2cm}
		
		\begin{example}\cite[Example 4.8]{ADRS}
			Consider the map germs $F:\mathbb{R}^{4} \rightarrow \mathbb{R}^{3}$ and $G:\mathbb{R}^{3}\rightarrow \mathbb{R}^{2}$ given by $F(x,y,z,w)=(x,y,z(x^{2}+y^{2}+z^{2}+w^{2}))$ and $G(u,v,t)=(ut,vt)$. Consequently, we have $H(x,y,z,w)=(xz(x^{2}+y^{2}+z^{2}+w^{2}),yz(x^{2}+y^{2}+z^{2}+w^{2}))$. One can easily show that $F$ and $G$ hold Milnor conditions  (b) and have isolated critical value; $\Sing G=\{t=0\}$; $H$ has isolated critical value and  $M(H)=\{z=0\}\cup\{w=0,z^{2}=x^{2}+y^{2}\}$. Moreover, one can show $\overline{F(M(H)\m \Sing H)}=\{t^2 = 4(u^2+v^2)^3\}$. Therefore, $  H $ satisfies the condition  (\ref{regularity condition}), and the Theorem \ref{tp} ensure that $H$ satisfies the Milnor condition (b).
			
		\end{example}

		\vspace{0.2cm}

		Some consequences of condition $(\ref{regularity condition})$ are Lemma 4.1 in \cite{DA} and Theorem 3.2 in \cite{CT23}. For more details, see  \cite[Corollary 4.6]{ADRS} and \cite[Corollary 4.7]{ADRS}. 
		
		\vspace{0.2cm}
		
		The next two characterizations are immediate consequences of Theorem \ref{tp}. 
		
		\vspace{0.2cm}
		
		\begin{theorem} \cite[Proposition 4.14]{ADRS} \label{incl}
			Let $F:(\mathbb{R}^m, 0) \to (\mathbb{R}^p,0)$ and $G:(\mathbb{R}^p, 0) \to (\mathbb{R}^k,0), m\ge p \ge k \ge 2$, be analytic map germs, with $F$ satisfying the Milnor condition  (b) and $ \Disc F =\{0\} $. If $F(M(H))\subseteq M(G)$ as a germ of sets and $G$ satisfies the Milnor condition (b), then the composition $H=G\circ F$ satisfies the Milnor condition (b).    
		\end{theorem}
		
		\vspace{0.2cm}
		
		\begin{example}
			Let $F:\bR^{4}\rightarrow\bR^{3}$ and $G:\bR^{3}\rightarrow \bR^{2}$  be two analytic maps germs given by $F(x,y,z,w)=(xw,yw,zw)$ and $G(u,v,t)=(u,v(u^{2}+v^{2}))$. Consequently, $H(x,y,z,w)=(wx, yw^{3}(x^{2}+y^{2}))$. It is easy to check out that $F, G$ hold Milnor condition  (b), and $F, G$ and $ H $ have isolated critical value. Moreover,  $M(G)=\{t=0\}\cup\{u=v=0\}$ and $M(H)=\{w=0\}\cup\{x=y=0\}\cup\{z=0,w^{2}=x^{2}+y^{2}\}$. Since $F(M(H))\subseteq M(G)$, by Theorem \ref{incl} one concludes that $H$ satisfies Milnor condition (b).
		\end{example}
		
		\vspace{0.2cm}
		
		\begin{corollary}\cite[Corollary 4.15]{ADRS}\label{rec}
			Let $F:(\mathbb{R}^m, 0) \to (\mathbb{R}^p,0)$ and $G:(\mathbb{R}^p, 0) \to (\mathbb{R}^k,0), m\ge p \ge k \ge 2$, be analytic map germs, such that $F$ satisfies the Milnor condition (b) and has isolate critical value. Suppose that $F(M(H))=M(G)$ as a germ set at the origin. Then $G$ satisfies the Milnor condition (b) if and only if the composition $H=G\circ F$ satisfies the Milnor condition (b).    
		\end{corollary}
		
		\vspace{0.2cm}
		
		We can find interesting applications associated with the Milnor condition (b) study for composition maps. The existence of the Milnor tube fibration for composition provides a powerful tool for relating the geometric and topological data for the Milnor fiber of $H$ to those of the Milnor fibers of $F$ and $G$, which shows that Milnor's conjecture is not only valid in the case where $G$ is the canonical projection, but in a more general case as was discussed in \cite{ADRS,CT23}. Furthermore, in \cite{ADRS}, the authors exhibit formulas that relate the Euler characteristic of the corresponding Milnor Fibers and Milnor tubes. See \cite{ADRS} for more details.

			\vspace{0.2cm}
			
			\subsection{Comparing Thom Regularity at $V_{G}$ and the Milnor Condition (b)}

			\vspace{0.5cm}
			
			We stated that the transversality of the fibers of a map $G$ to the levels of a function $\rho$, which defines the origin, i.e., the \textit{$\rho$-regularity} of $G$, is a condition for the existence of local fibration structures. We also saw that the Thom regularity conditions of $V_G$ and the Milnor condition (b) at the origin imply the $\rho$-regularity of $G$. Therefore, these conditions have a close relationship with the existence of fibration structures such as the Milnor fibration over the tube and spheres. This connection justifies the need to create strategies to find natural and convenient conditions to verify Thom regularity and the Milnor condition (b) and to compare these two conditions to build new and interesting classes of maps $G$ with such conditions, as well as to study the singularities of already known classes.
			
			\vspace{0.2cm}
			
			Our next result helps to understand the relationship between the Thom regularity of a map $G$ and the Milnor condition (b) at the origin.

			\vspace{0.2cm}
			
			\begin{proposition}\label{ppp}
				Let $G: (\mathbb{R}^m, 0) \to (\mathbb{R}^p,0)$ be a real analytic germ. If $G$ is Thom regular in $V_G$, then $G$ satisfies the Milnor condition (b).
			\end{proposition}
			
			\begin{proof}
				It directly follows from Remark \ref{obs1} and Lemma \ref{f2L3}.
			\end{proof}
			
			\vspace{0.2cm}
			
			Therefore, Thom regularity implies the Milnor condition (b) at the origin. However, the converse of this result is not valid, as shown by the authors in \cite{ACT}. In this paper, the authors presented Example \eqref{T} below, which satisfies the Milnor condition (b) at the origin but does not satisfy Thom regularity. This example demonstrated that the conditions are not equivalent; in particular, it also showed that existence theorems for fibration structures based on the Milnor condition (b) are more general than those based on the Thom regularity of the map. Finally, in \cite{PT} and \cite{R}, the authors discussed the size of the class of maps satisfying the Milnor condition (b) concerning the class of maps satisfying Thom regularity, and in \cite{R}, the author constructed an infinity of new examples of real analytic maps that have Milnor fibrations over the tube without Thom regularity.
			
			\vspace{0.2cm}
			
			The example below was the first example of an analytic function that has a Milnor fibration over the tube without satisfying Thom regularity in $V_G$, showing that the converse of Proposition \ref{ppp} is not true in general. This example was developed by Mihai Marius Tib\u ar and presented in the article \cite{ACT}. Subsequently, other examples of such mixed functions, were presented by Mutso Oka, Anne Pichon, José Seade, and Maico Ribeiro.
			
			\vspace{0.2cm}
			
			\begin{example}\label{T}
				$T:\bC^2\to \bC$ given by $T(x,y)=xy\bar{x}$.
			\end{example}
			
			\vspace{0.2cm}
			
			In the article \cite{PT}, the authors found a new family of mixed functions satisfying the Milnor (b) condition without satisfying Thom regularity. The family is presented in the example below.
			
			\vspace{0.2cm}
			
			\begin{example}\label{pt}
				$P_k:\bC^3\to \bC$ given by $P_k(x,y,z)=(x+z^k)\bar{x}y$, with $k\ge 2$.
			\end{example}
			
			\vspace{0.2cm}
			
			The next result, proved in the article \cite{R}, provides a mechanism for constructing an infinity of mixed functions satisfying the Milnor condition (b)  without satisfying Thom regularity. 
			
			\vspace{0.2cm}
			
			\begin{theorem}\cite{R}\label{ttt}
				Let $f:(\mathbb{C}^n,0) \to (\mathbb{C},0)$ be mixed functions with isolated critical value, without Thom regularity. Let  
				$g:(\mathbb{C}^m,0) \to (\mathbb{C},0)$ be an horizontally weakly conformal map such that $f$ and $g$ have separable variables, and $G=f+g$  satisfies the Milnor condition (b).
				Then, $G$  has Milnor tube fibration without Thom regularity.
			\end{theorem}

			\vspace{0.2cm}
			
			\begin{example}
				Consider the following mixed functions:
				$f(x,y)=xy^k \bar{x}$, for a fixed $k\geq 1$. 	One has that $f$ has isolated critical value and it is not Thom regular at $ V_{f} $, see \cite{R} for more details.  	Now, consider the mixed function $G:=f + g$ where $$g(z,w)=z^{k+n}\bar{w}^n$$ is  horizontally weakly conformal map. One has that $ G(x,y,z,w)=xy^k\bar{x}+z^{k+n}\bar{w}^n $ is a polar weighted-homogeneous of type $ (1,1,k+n,n;k) $, for any $ k,n\ge 1 $, consequently,  $ G $ satisfies the Milnor condition (b), see \cite{ACT} for more details. Since $ f $ and $ g $ have separable variables, it follows from   Theorem \ref{ttt}  that $ G $ has Milnor tube fibration without Thom regularity. 
			\end{example}
			
			\vspace{0.3cm}

			Now, we recall that if $ A $ is a subset of a topological space $ X, $ the \textit{interior} of $ A $ is defined as the union of all open sets of $ X $ that are contained in $ A $. Therefore,  $ A $ has an \textit{empty interior in $ X $} if $ A $ contains no open set of $ X $ other than the empty set. Equivalently, $ A $ has an empty interior if every point of $ A $ is a limit point of the complement of $ A$, i.e., if the complement of $ A $ is dense in $ X $, see \cite{Mu}.
			
			\vspace{0.3cm}
			
			Next, if $ X $ is a topological space and  $ A \subset X$, we consider the set $ {Bd(A)}:= (\overline{X \m A}) \cap A.$ Then, $Int(A) \textrm{  and  } {Bd(A)}$ are disjoint, and $ \bar{A} = Int(A) \cup {Bd(A)}. $ 
			
			\vspace{0.3cm}
			
			The next proposition is the real counterpart to \cite[Proposition 10]{RSR}.
			
			\vspace{0.3cm}
			
			\begin{proposition}\label{f2prop2}
				Let $G:(\mathbb{R}^m,0) \to (\mathbb{R}^p,0)$ be a germ of an analytic map with an isolated critical value such that $V_G \subset M(G)$. If $V_G$ has an empty interior in $M(G)$ and $\textrm{dim}V_G>0$, then $ G $ is not Thom $(\aa_G)$-regular at $V_G$.
				
			\end{proposition}
			\begin{proof}
				If the complement of $ V_G $ is dense in $ M(G) $, then $ V_G \subset  \overline{M(G) \m V_G} $, which implies $ \overline{M(G) \m V_G} \cap V_G = V_G$. In general, since $ V_G $ has an empty interior in $ M(G) $, then $ Int(V_G) = \emptyset $. It follows from the previous discussion  that $\overline{V}_G = {Bd(V_G)}$. Now, since $ V_G $ 
				is closed we get $ \overline{M(G) \m V_G} \cap V_G = V_G$. Agora basta usar a Proposição \ref{ppp}.
			\end{proof}
			
			\vspace{0.3cm}
			
			\begin{remark} 	Consider $G (x, y, z) = (x, y (x ^ 2 + y ^ 2) + xz ^ 2)$ in three real variables. One has that $V=\left\lbrace x=y=0\right\rbrace $ and $M (G) = \left\lbrace z=0\right\rbrace \cup \left\lbrace x=y=0\right\rbrace $. Hence, $\overline{M(G)\m V} \cap V = \left\{0\right\}$. This example does not contradict the Proposition \ref{f2prop2} since $ V_G $ does not have an empty interior in $ M(G)$. \end{remark}
			
			\vspace{0.2cm}
			
			\begin{example}\label{e1}
				Consider the mixed map $ G:\mathbb{C}^3 \to \mathbb{C}^2$ given by $G(x,y,z)= (x^2 - y^2z,y) $. By straightforward computations, one can see that $V=\Sing G =\left\lbrace x=y=0 \right\rbrace$ and $ \left\lbrace x=0\right\rbrace \subset M(G). $ Consequently, $ V_G $ has an empty interior in $ M(G). $ Since $\textrm{dim} V_G >0 $, it follows from Proposition \ref{f2prop2} that $ G $ does not satisfy the condition \eqref{eq:mainclass} and does not have the Thom regularity condition.	
			\end{example}

	\end{document}